\documentclass[12pt, reqno]{amsart}

\author[M.~Caprio and S.~Mukherjee]{Michele Caprio and Sayan Mukherjee}
\address{PRECISE Center, Dept. of Computer and Information Science,
 University of Pennsylvania, 3330 Walnut Street, Philadelphia, PA}
\email{caprio@seas.upenn.edu}
\urladdr{\url{https://michelecaprio.wixsite.com/caprio}}   
\address{Department of Computer Science, Universität Leipzig, Paulinum
Augustusplatz 10, Leipzig, Germany 04109}
\email{sayan.mukherjee@mis.mpg.de}
\urladdr{\url{https://sayanmuk.github.io/}}


\title{Extended probabilities and their application to statistical inference}

\usepackage[T1]{fontenc}
\usepackage{amsmath}
\usepackage{amssymb}
\usepackage{amsthm}
\usepackage[left=2.5cm, right=2.5cm, top=3cm]{geometry}
\usepackage{hyperref}
\usepackage{algorithm}
\usepackage{algpseudocode}
\usepackage{bbm}
\usepackage{tikz}
\usepackage{caption}
\usepackage{subcaption}
\usepackage{fancyhdr}
\usepackage[inline]{enumitem}
\usepackage{comment}
\usepackage{nicefrac}
\usepackage{bm}
\usepackage{mathrsfs}
\usepackage{graphicx}
\usepackage[utf8]{inputenc}
\usepackage{cancel}
\usepackage{mathtools}

\newcommand{\vertiii}[1]{{\left\vert\kern-0.25ex\left\vert\kern-0.25ex\left\vert #1 
    \right\vert\kern-0.25ex\right\vert\kern-0.25ex\right\vert}}
		
\AtBeginDocument{%
   \def\MR#1{}
}

\newtheorem{thm}{Theorem}[section]

\theoremstyle{definition} 
\newtheorem{defi}[thm]{Definition}
\let\olddefi\defi
\renewcommand{\defi}{\olddefi\normalfont}
\newtheorem{rmk}[thm]{Remark}
\let\oldrmk\rmk
\renewcommand{\rmk}{\oldrmk\normalfont}

\newtheorem{theorem}{Theorem}

\newtheorem{proposition}[theorem]{Proposition}
\newtheorem{corollary}[theorem]{Corollary}

\newtheorem{definition}[theorem]{Definition}
\newtheorem{remark}[theorem]{Remark}
\newtheorem{example}[theorem]{Example}

\hypersetup{
    pdftitle={prove},
    pdfauthor={autori},
    pdfmenubar=false,
    pdffitwindow=true,
    pdfstartview=FitH,
    colorlinks=true,
    linkcolor=blue,
    citecolor=green,
    urlcolor=cyan
}

\providecommand{\MR}[1]{}

\providecommand{\MR}{\relax\ifhmode\unskip\space\fi MR }

\providecommand{\href}[2]{#2}

\newcommand\xqed[1]{%
  \leavevmode\unskip\penalty9999 \hbox{}\nobreak\hfill
  \quad\hbox{#1}}
\newcommand\demo{\xqed{$\triangle$}}

\begin{document}

\maketitle
\thispagestyle{empty}

\begin{abstract}
We propose a new, more general definition of extended probability measures. We study their properties and provide a behavioral interpretation. We put them to use in an inference procedure, whose environment is canonically represented by the probability space $(\Omega,\mathcal{F},P)$, when both $P$ and the composition of $\Omega$ are unknown. We develop an ex ante analysis -- taking place before the statistical analysis requiring knowledge of $\Omega$ -- in which the true composition of $\Omega$ is progressively learned. We describe how to update extended probabilities in this setting, and introduce the concept of lower extended probabilities. We apply our findings to a species sampling problem and to the study of the boomerang effect (the empirical observation that sometimes persuasion yields the opposite effect: the persuaded agent moves their opinion away from the opinion of the persuading agent).
\end{abstract}

\section{Introduction}
Researchers use the terminology ``extended probabilities'' to refer to set functions whose codomain is either a superset of $[0,1]$, or defined using entirely different number types, such as $p$-adic numbers  \cite{khrennikov}. They first came up in physics (noticed by \cite{dirac} and \cite{heisenberg}), where they are still studied today \cite{ferrie,hartle2,kronz}. They then became popular in other fields too, including economics and finance \cite{burgin3, jarrow}, machine learning \cite{lowe}, stochastic processes \cite{wiewiora}, and queuing theory \cite{tijms}. Of course, they have been extensively studied in mathematics (see e.g. \cite{allen} and \cite{bartlett} for early works, and \cite{burgin2} and \cite{khrennikov} for more recent ones). A complete account is given in \cite{burgin1}. 

``Extended probabilities'' have been given differing definitions and interpretations across -- and even within -- fields of study.  In quantum theory, for example, \cite{benavoli} point out that ``negative probabilities'' do not have intrinsic meaning beyond the fact that they constitute a probabilistic model compatible with quantum events.  In \cite{gell-mann, hartle1, hartle2}, instead,  the authors interpret them as being associated with unsettleable bets (see section \ref{rel_literature}).



In this paper, we aim to give a foundational definition of extended probabilities. We give two properties 
that a set function must have in order to be called an extended probability, irrespective of the field the scholar works in. This represents a great improvement with respect to previous works on the matter, in which definitions depend on the area of study. 

We explain how extended probabilities are different from regular Kolmogorovian probabilities, and we characterize some of their more interesting properties. We also give a behavioral interpretation of extended probabilities taking on negative values that complements the frequentist one given in \cite{burgin1} to ``negative probabilities''. Finally, we relate our definition and interpretation to the ones existing in the literature, and we illustrate how these latter can be reconciled with the ones we provide.

After that, we present what is to the best of our knowledge the first application of extended probabilities to an inference procedure. Consider a generic statistical experiment; let $\Omega$ be a finite or countable space, and endow it with the sigma-algebra $\mathcal{F} = 2^\Omega$. The space usually adopted to express uncertainty around the elements of $\mathcal{F}$ is the probability space $(\Omega,\mathcal{F},P)$, for some probability measure $P:\mathcal{F} \rightarrow [0,1]$. Now, suppose we want to express further uncertainty regarding either the composition of $\Omega$, or which $P$ to consider on $(\Omega,\mathcal{F})$; we can do so by using lower probabilities. In particular, in the first case, we consider the probability space $(\Omega,\mathcal{F},P)$ and we follow the example in \cite{gong2}. There, as a consequence of survey nonresponse (that is, subjects not answering to the questions in a survey), the author is forced to consider $\check{\Omega}=2^\Omega$. In this case, probabilities cannot be computed exactly: only lower and upper bounds to precise probabilities -- lower and upper probabilities, respectively -- are available. In the second case, we consider the triple $(\Omega,\mathcal{F},\mathcal{P})$, where $\mathcal{P}$ is a set of probability measures, and we proceed e.g. as in \cite{prob.kin,prob.kin.ergo}.

To the best of our knowledge, there is no cogent way of expressing uncertainty on {both} the composition of $\Omega$ and which $P$ to consider on $(\Omega,\mathcal{F})$. In this work, we aim to fill this gap by using extended probabilities. We describe an ex ante analysis, meaning one that takes place before the actual statistical analysis that requires the knowledge of the state space. In the most general case, we start from the number type we believe we are working with (naturals, wholes, integers, or rationals), and call it $\Omega$. Then, at time $0$ we divide it into an actual space $\Omega^+_0$ that we deem a plausible state space for our experiment, and a latent one $\Omega_0^-$, which we do not know about; notationally, $\Omega=\Omega_0^- \sqcup \Omega_0^+$, where $\sqcup$ denotes a disjoint union of sets. We assign negative extended probabilities to the subsets of $\Omega_0^-$. To capture the uncertainty around which $P$ to consider, we specify a set $\mathcal{P}^{ex}$ of extended probabilities supported on the whole $\Omega$, instead of a single one. We then describe how we progressively discover the true composition of the state space associated with our experiment (which may be the whole $\Omega$ we initially specified, or a proper subset $\Omega^\prime \subsetneq \Omega$), and how to update extended probabilities accordingly. We conclude our analysis by discovering the  state space associated with our experiment. 

In addition, we show that the limiting set of the sequence $(\mathcal{P}^{ex}_t)$ of updates of the initially-specified set of extended probability measures must be one of the following. It is either a set of regular probability measures, if the state space associated with our experiment is the whole $\Omega$, or a set of extended probability measures if the state space associated with our experiment is a proper subset $\Omega^\prime$ of $\Omega$. In this latter case, the limiting set induces a set of regular probability measures on $\Omega^\prime$.

We also develop the concept of lower and upper extended probabilities. They represent the ``boundaries'' of a generic set $\mathcal{P}^{ex}$ of extended probabilities, and more in general allow for an imprecise elicitation of extended probabilities. They are extremely important because under a mild assumption, knowing lower probability $\underline{P}^{ex}$ is enough to be able to retrieve the whole set $\mathcal{P}^{ex}$. We provide bounds for the lower extended probability $\underline{P}^{ex}_t(A)$ of any element $A \in \mathcal{F}$, at any time $t$ in our ex ante analysis. For the sake of completeness, we  give the definition of an extended Choquet capacity, and we show how lower and upper extended probabilities are extended Choquet capacities.

In Example \ref{ex_species}, we provide a simple application to the field of ecology. We illustrate how the analysis we describe in the paper can be put to use in a species sampling problem, specifically to retrieve the number of birds that inhabit a certain region throughout the year.


We also provide an example -- adapted from the one in \cite{allah} -- in the field of opinion dynamics; in particular, we describe the boomerang effect. We have a persuading agent acting on a persuaded agent, but the latter perceives the former as having low credibility. This can be modeled so that the persuaded agent does not know the composition of the entire state space, while she suspects the persuading agent does: she thinks the persuading agent may be hiding something form her. As she discovers the true composition of the state space, the credibility of the persuading agent is restored.

This paper is organized as follows: in section \ref{extended} we give the foundational definition of an extended probability, its properties, and its behavioral interpretation. In section \ref{main} we use extended probabilities in an inference procedure. Section \ref{ulep} deals with lower extended probabilities, and section \ref{conclusion} concludes our work. In appendix \ref{proofs}, we give the proofs of our results. In appendix \ref{opdin} we give the opinion dynamics example, and in appendix \ref{definetti} we report two interesting quotes from \cite{definetti1}.

\begin{remark}\label{algebra}
To deal with uncertainty in the composition of $\Omega$, one could proceed as in \cite[section 4.3.3]{walley}. The agent could begin the elicitation by specifying the state space $\Omega_0$ and then, as they analyze the problem in greater detail, could realize that a refinement to a finer-grained $\Omega_1$ is needed. This corresponds to specifying what \cite{dempster} calls a multivalued mapping from $\Omega_0$ to $\Omega_1$; that is,
$$A:\Omega_0 \rightrightarrows \Omega_1, \quad \omega_0 \mapsto A(\omega_0) \subset \Omega_1.$$
So $\Omega_0$ corresponds to a partition of $\Omega_1$, and each state $\omega_0 \in \Omega_0$ can be identified with the set $A(\omega_0)$ of ``refined possibilities'' in $\Omega_1$. To illustrate the complication deriving from this approach, let $\Omega_0$ and $\Omega_1$ be finite or countable, and call $\mathcal{F}_0=2^{\Omega_0}$ and $\mathcal{F}_1=2^{\Omega_1}$. If the agent specifies a probability measure $P_0$ on $(\Omega_0,\mathcal{F}_0)$, they then need to come up with a probability measure $P_1$ on $(\Omega_1,\mathcal{F}_1)$ such that
$$P_0(\{\omega_0\})=\sum_{\omega_1 \in A(\omega_0)} P_1(\{\omega_1\})$$
holds for all $\omega_0 \in \Omega_0$. This means that a new (subjective) probability elicitation must take place once the agent refines the state space to $\Omega_1$. This can be avoided using extended probabilities, as we shall argue in section \ref{main}. 

It is worth noting that while this multivalued mapping approach is pointed out in \cite{dempster,walley}, the authors sidestep its associated complications by means of probability bounding. Furthermore, in \cite{walley_bom} too the author tries to solve the incompletely specified possibility space problem by probability bounding; this enables progressive learning of the sample space in a way that  representation invariance as well as symmetry are satisfied. We will verify whether the model that we introduce in section \ref{main} meets these conditions in future work.
\end{remark}

\section{Extended probability measures}\label{extended}

In this section, we first illustrate the philosophical reason to introduce extended probabilities, and then we dive into more technical details. We conclude with a thorough analysis on how our interpretation of extended probabilities relates to the existing literture.
\subsection{Philosophical motivation for extended probabilities}\label{philo}
We give a behavioral interpretation of extended probabilities. 
Consider a generic event $A$, and suppose we want to express our belief about the likelihood of it taking place. Suppose we can enter a bet about $A$ that gives us \$$1$ if event $A$ happens and \$$0$ if it does not happen. Then, we say that the probability we attach to $A$ is given by
$p \geq 0$, 
the supremum buying price as well as the infimum selling price for the bet about $A$ (we call it \textit{fair price}). That is, for every $\varepsilon>0$, we are willing to enter a bet which gives $1-p+\varepsilon$ if $A$ happens, and $-p+\varepsilon$ if $A$ does not happen, as well as the bet which gives $p-1+\varepsilon$ if $A$ happens and $p+\varepsilon$ if $A$ does not happen. 
This interpretation is inspired by the classical subjective probability interpretation given by de Finetti in \cite{definetti1,definetti2}. As pointed out in \cite{nau}, 
the other two fathers of subjective probability theory, Ramsey \cite{ramsey} and Savage \cite{savage} simultaneously introduced measurement schemes for utility. They tied their definitions of probability to bets in which the payoffs were effectively measured in utiles rather than dollars. In this way, they obtained probabilities that were interpretable as measures of pure belief, uncontaminated by marginal utilities for money. De Finetti later admitted that it might have been better to adopt the seemingly more general approach of Ramsey and Savage, since it leads to a theory of decision-making that does not rely on monetary values. Nevertheless, he found other reasons for preferring the money bet approach. In particular, he maintained that it would be extremely difficult to settle bets based on utiles, because the monetary sums needed to settle them would need to be adjusted to the complex variations in a unit of measure (utiles) that is unobservable.\footnote{The complete quotes from \cite{definetti1} can be found in Appendix \ref{definetti}.} This is why 
we retain the DeFinettian interpretation of (subjective) probability, and more in general why we adopt a betting scheme approach. 

Suppose now that we are not given the possibility to enter a bet like the aforementioned one, so we cannot assess a probability for $A$ as before. Instead, such a possibility is given to our doppelgänger, who tells us that their subjective assessment for the probability of $A$ is some $q \in [0,1]$. Here, it is assumed that the doppelgänger assigns the same 
probabilities as ourselves to all the events we both can enter a bet about. 
This procedure of asking the doppelgänger is equivalent to setting the probability of $A$ ourselves; we introduce the doppelgänger -- a logical artifice -- because, given our interpretation of probability, it is impossible to elicit the probability of event $A$ if we cannot enter a bet similar to the one described before. We conclude that if we were given the opportunity to enter the bet about $A$, the amount of money we would deem fair to pay would be $q$ dollars. Therefore we are prepared to lose \$$q$ in the case $A^c$ happens.\footnote{Provided that we gain \$$(1-q)$ if $A$ happens.} We express this by setting $p=-q$.


If after a while we are given the opportunity to enter a bet about an event $A$ whose extended probability we deemed to be $-q<0$, the price we consider fair to pay need not be $p=|-q|$. If $p\neq |-q|$, it means that we changed idea about how likely event $A$ is once given the possibility of entering the bet. If instead $p=|-q|$, we are obeying to Allen's {principle of conservation of knowledge} \cite{allen}. It states that, like the law of conservation of mass, knowledge is neither created nor destroyed, but only transformed; its total possible amount is constant. By having us choose the probability equal in absolute value, we comply with conservation of knowledge. As we can see, this principle allows us to mechanically retrieve positive probabilities starting from negative ones. It also allows us to work backwards to negative probabilities starting from positive ones. Suppose we have an event $A \in \mathcal{F}$ to which we attach probability $p\geq 0$. 
Then, we know that if we were not given the opportunity to enter the bet that allowed us to indicate $p$ as the probability that event $A$ happens, then we would have expressed our uncertainty by assigning $A$ probability $-p\leq 0$.

A {Dutch book} is, informally, the possibility of constructing a bet such that the bookmaker always profits, while the punter always loses money (see Remark \ref{rem_dutch} for a formal definition). In \cite{definetti3}, a {coherent} subjective probability is defined as one that does not allow for a Dutch book to be made against the punter, however a bet is made. Necessary and sufficient conditions for coherence require that subjective probabilities satisfy the Kolmogorovian axioms of probability (with only finite additivity). As we shall see in section \ref{technical}, this does not hold for extended probabilities. This is not a problem though, in light of the interpretation we give to negative probabilities. The events whose attached probabilities are negative are events for which the punter cannot enter a bet; hence, they cannot be used to build a Dutch book. We give the definition of coherence in the context of extended probabilities and the formal statement that extended probabilities are always coherent in Remark \ref{rem_dutch}. In addition, as we show in section \ref{main}, the inferential procedure we consider is such that, starting with extended probabilities, we recover -- at the end of our analysis -- regular probabilities. This also ensures that no Dutch books can be created.

\subsection{Technical definition and properties}\label{technical}
Consider a measurable space $(\Omega,\mathcal{F})$, where $\mathcal{F}\subset 2^\Omega$ is a sigma-algebra.  
An extended probability $P^{ex}:\mathcal{F} \rightarrow \mathbb{R}$ is a set function on $\mathcal{F}$ such that
\begin{itemize}
\item[(i*)] $P^{ex}(A) \in [-1,1]$, for all $A \in \mathcal{F}$;
\item[(ii*)] if $\{A_j\}_{j \in I}$ is a countable collection of disjoint events such that $\cup_{j \in I} A_j \in \mathcal{F}$, then
$$P^{ex} \left( \bigcup\limits_{j \in I} A_j \right) = \sum\limits_{j \in I} P^{ex}(A_j);$$
\item[(iii*)] given any partition $\mathcal{E}=\{E\}$ of $\Omega$ such that $E\in\mathcal{F}$, for all $E\in\mathcal{E}$, $\sum_{E\in\mathcal{E}}|P^{ex}(E)|=1$.
\end{itemize}
Condition (iii*) is needed in light of the interpretation we gave in section \ref{philo} to extended probabilities: if (iii*) were not to hold, the doppelgänger's opinion would be represented by a signed measure, not a probability measure.

Let $\mathcal{F}\subset 2^\Omega$ be the same sigma-algebra as before. The Kolmogorovian axioms for any regular probability measure $P:\mathcal{F} \rightarrow \mathbb{R}$ are the following
\begin{itemize}
\item[(K1)] $P(A) \in [0,1]$, for all $A \in \mathcal{F}$;
\item[(K2)] if $\{A_j\}_{j \in I}$ is a countable collection of disjoint events such that $\cup_{j \in I} A_j \in \mathcal{F}$, then $P ( \cup_{j \in I} A_j ) = \sum_{j \in I} P(A_j)$;\footnote{We consider the countably additive version of Kolmogorov axioms.}
\item[(K3)] $P(\Omega)=1$.
\end{itemize}
It is easy to see, then, how conditions (i*) and (iii*) are more general than (K1) and (K3), respectively. To this extent, extended probabilities are a generalization of the concept of regular probabilities. From its definition, we see that an extended probability is a finite signed measure. We call the triple $(\Omega,\mathcal{F},P^{ex})$ an \textit{extended probability space}.

Because we do not require $P^{ex}(\Omega)=1$, the extended probability of the complement of an event $A$ is given by
\begin{align}\label{eq1}
P^{ex}(A^c)=P^{ex}(\Omega \setminus A)=P^{ex}(\Omega)-P^{ex}(A).
\end{align}
Equation (\ref{eq1}) comes from (ii*); indeed, consider the disjoint sets $\Omega \setminus A$ and $\Omega \cap A=A$. Then, 
\begin{align*}
P^{ex}(\Omega)&=P^{ex}\left( \left[\Omega \setminus A \right]\sqcup A \right) =P^{ex}(\Omega \setminus A)+P^{ex}(A)\\ &\iff P^{ex}(\Omega \setminus A)=P^{ex}(\Omega)-P^{ex}(A).
\end{align*}

Equation (\ref{eq1}) ensures us that $P^{ex}(\emptyset)= P^{ex}(\Omega^c) = P^{ex}(\Omega)-P^{ex}(\Omega)=0$. We also have that $P^{ex}(\Omega)=\sum_{E_j \in \mathcal{E}} P^{ex}(E_j)$, where $\mathcal{E}=\{E_j\}$ is any finite or countable partition of $\Omega$. 

\begin{proposition}\label{prop1}
Extended probabilities have the adequacy property: for all $A,B \in \mathcal{F}$ such that $A=B$, then $P^{ex}(A)=P^{ex}(B)$. 
\end{proposition}

We also have that extended probabilities retain a version of the monotonic property of regular probabilities.

\begin{proposition}\label{prop2}
For all $A,B \in \mathcal{F}$ such that $A \subset B$, if $P^{ex}(A) \geq 0$, $P^{ex}(B) \geq 0$, and $P^{ex}(B\cap A^c) \geq 0$, then $P^{ex}(A) \leq P^{ex}(B)$. If instead $P^{ex}(A) \leq 0$, $P^{ex}(B) \leq 0$, and $P^{ex}(B \cap A^c) \leq 0$, then $P^{ex}(A) \geq P^{ex}(B)$.
\end{proposition}

This version of the monotonic property implies a version of the continuity property enjoyed by regular probabilities.
\begin{corollary}\label{cor2.1}
If we have a collection $\{A_j\}$ of elements of $\mathcal{F}$ such that
\begin{itemize}
\item[\textbullet] it is nested (i.e. $A_1 \supset A_2 \supset \cdots \supset A_n \supset \cdots$),
\item[\textbullet] $\cap_j A_j=\emptyset$,
\item[\textbullet] $P^{ex}(A_j)\geq 0$, for all $A_j$,
\end{itemize}
then $\lim_{j \rightarrow \infty} P^{ex}(A_j)=P^{ex}(\cap_j A_j)=0$. 
\end{corollary}

Extended probabilities satisfy the inclusion-exclusion principle.
\begin{proposition}\label{prop3ext}
The following is true: $$P^{ex}(A \cup B)=P^{ex}(A)+P^{ex}(B)-P^{ex}(A\cap B),$$
for all $A,B \in \mathcal{F}$.
\end{proposition}

This implies immediately that, for any $A,B,C \in \mathcal{F}$, 
\begin{align*}
P^{ex}((A \cup B)\cap C)&=P^{ex}((A \cap C)\cup (B \cap C))\\&=P^{ex}(A\cap C)+P^{ex}(B\cap C)-P^{ex}(A\cap B\cap C),
\end{align*}
by Proposition \ref{prop3ext} and the De Morgan laws.

We define the extended conditional probability of $A$ given $B$ to be the extended counterpart of regular conditional probabilities,
\begin{align}\label{eq13}
P^{ex}(A \mid B):=\frac{P^{ex}(A\cap B)}{P^{ex}(B)} \in [-1,1],
\end{align}
for all $A,B \in \mathcal{F}$ such that $P^{ex}(B) \neq 0$.\footnote{Notice that in de Finetti's approach \cite{definetti1,definetti2} a similar formula follows from the ``called off bet'' interpretation of (regular) conditional probability. In the future, we plan to provide a direct behavioral motivation for \eqref{eq13}, in the spirit of de Finetti's work.} 

To avoid confusion arising from the sign, we say that $A,B \in \mathcal{F}$ are independent if and only if
$$|P^{ex}(A\cap B)|=|P^{ex}(A)\times P^{ex}(B)|.$$
This way of describing independence corresponds to the one given in \cite[section 2]{allen}.

\begin{remark}\label{rem_dutch}
Consider any extended probability space $(\Omega,\mathcal{F},P^{ex})$. Call $\mathscr{B}$ the sigma-algebra generated by the events $B$ in $\Omega$ that we can enter a bet about, that is, the sigma-algebra generated by the collection $\{B \subset \Omega: P^{ex}(B)\geq 0 \text{ and } B\cap C = \emptyset \text{, } \forall C \in\mathcal{F} \text{, } P^{ex}(C)<0\}$. Then, a \textit{Dutch book} is a finite collection $\{B_j\}_{j=1}^n \subset \mathscr{B}$ along with numbers $\{s_j\}_{j=1}^n \subset \mathbb{R}$ such that, 
\begin{align}\label{eq27}
\sup_{\omega\in\Omega} \mathfrak{f}(\omega):=\sup_{\omega\in\Omega} \left\lbrace{\sum_{j=1}^n s_j \left[ \mathbbm{1}_{B_j}(\omega)-P^{ex}(B_j) \right]}\right\rbrace<0.
\end{align}
Notice that the $s_j$'s are the payout for a winning bet on $B_j$, and $s_jP^{ex}(B_j)$ is the bet's fair buy-in. The inequality in \eqref{eq27} suggests that a bet can be built such that the bookmaker always gets a profit, while the punter always loses money. 
\begin{definition}\label{coherence}
Extended probability $P^{ex}$ is coherent if no Dutch books can be made against the punter.
\end{definition}
Then, we have the following important result.
\begin{theorem}\label{alw_coh}
$P^{ex}$ is always coherent.
\end{theorem}

In our definition of Dutch book, we do not take into account the events $C \in \mathcal{F}$ for which $P^{ex}(C) < 0$ since those are events for which the punter cannot enter a bet. Hence, they cannot be used to build a Dutch book. 
\end{remark}

\subsection{Related literature}\label{rel_literature}
Let us now inspect how the definition and the interpretation of extended probabilities we have given so far relates to the existing literature.

Our framework is related to the one studied in \cite{burgin1}. Burgin studies a static environment (as opposed to the dynamic one we inspect in this work in the next section) where the state space $\Omega$ can be divided in two irreducible parts $\Omega^+$ and $\Omega^-$ such that $\#\Omega^+=\#\Omega^-$, where $\#$ denotes the cardinality operator. The elements of $\Omega^-$ are called anti-events, and they are usually connected to negative objects: encountering a negative object is a negative event. The author then states that an example of a negative object is given by antiparticles, the antimatter counterpart of quantum particles. Anti-events are given negative probabilities. If we require that the following conditions hold, we obtain a framework very similar to Burgin's one:
\begin{itemize}
\item[\textbullet] extended probabilities are {finitely} additive (instead of countably additive);
\item[\textbullet] our state space can be divided in two irreducible parts $\Omega^+$ and $\Omega^-$;
\item[\textbullet] $P^{ex}(A) \geq 0$ if and only if $A \subset \Omega^+$;
\item[\textbullet] there exists a function $\alpha:\Omega \rightarrow \Omega$ such that $\alpha(\omega)=-\omega$ and $\alpha^2(\omega)=\omega$;
\item[\textbullet] $P^{ex}(\Omega^+)=1$;
\item[\textbullet] $\{v_i,\omega,-\omega: v_i,\omega \in \Omega \text{, } i \in I\} = \{v_i: v_i \in \Omega \text{, } i \in I\}$, for all $\omega \in \Omega$ and all set of indices $I$.
\end{itemize}
The main difference is that in \cite{burgin1} $P^{ex}(\Omega^+)+|P^{ex}(\Omega^-)|=2$. Our interpretation of negative probabilities reconciles with Burgin's one if we consider anti-events as events for which we cannot enter a bet, which seems a reasonable assumption. He also gives a frequentist interpretation of negative probabilities, that complements our interpretation of negative extended probabilities.

In a subsequent paper \cite{burgin2}, Burgin generalizes his own setup, mainly by not requiring $\#\Omega^+=\#\Omega^-$ and by allowing any event to have either positive or negative probability, depending on external conditions. If we require that the following conditions hold, we obtain Burgin's generalized framework:
\begin{itemize}
\item[\textbullet] extended probabilities are finitely additive;
\item[\textbullet] $A \in \mathcal{F}$ implies $-A:=\{-\omega:\omega\in A\} \in \mathcal{F}$;
\item[\textbullet] for all $A \in \mathcal{F}$, $P^{ex}(-A)=-P^{ex}(A)$.
\end{itemize}
Notice that this last condition is very similar to Allen's principle of conservation of knowledge. The interpretation he gives for negative probabilities is similar to the one given in his previous work, with the peculiarity that now negative probabilities are not assigned exclusively to anti-events. Our interpretation can be seen as being a step behind Burgin's one: we give a specific reason for why an event $\tilde{A}$ is assigned a negative probability, namely that we cannot enter a bet on it. Then, Burgin states that the anti-event of $\tilde{A}$, $-\tilde{A}$, will have a positive (regular) probability. Of course, the vice versa holds: if an event is assigned a positive probability (because we can enter the bet), then according to Burgin its anti-event will have a negative probability.

Another interesting interpretation is given by Székely in \cite{szekeley}. The author proves that we can encounter a negative probability if we work with a random variable having a signed distribution. In addition, if $X$ has a signed distribution, then there exist two random variables $Y,Z$ having an ordinary (not signed) distribution such that $X+Y=Z$ in distribution. Therefore $X$ can be seen as a ``difference'' of two ordinary random variables $Z$ and $Y$. We reconcile our interpretation and Székely's one as follows. A signed distribution $P^s$ is simply an extended probability on the space of outcomes of a random variable. A pullback argument can then be used to define an extended probability $P^{ex}$ on $\Omega$: $P^{ex}(X^{-1}(I))=P^s(I)$, for all subsets $I$ of the outcome space of random variable $X$. So there are going to be events in $\Omega$ having negative probabilities: those are events we cannot enter a bet about.

The interpretation Kronz \cite{kronz} gives of negative probabilities is the following: he calls negative probabilities {inferred probabilities}, that can only be obtained indirectly by inference from operational (regular) probabilities. The associated events (that is, events having negative probabilities) are called {virtual events} in that they are non-operational, and so do not give rise to directly accessible relative frequencies. Actual events have non-negligible effect on them, and the reverse is also true. As it appears clear, virtual events are equivalent to latent events e.g. in the psychology \cite{bollen}, economics \cite{hu} and medicine \cite{rabe} literatures: Kronz assigns negative probabilities to latent events. These latter are events we do not observe, so it is fair to think we cannot enter a bet involving them. 

In \cite{wigner}, Wigner allows probabilities to go negative to sidestep the uncertainty principle of quantum mechanics. This latter states that given a particle that moves in one dimension, there is a limit to the precision with which position $x$ and momentum $p$ can be determined simultaneously in a given state. The observer can know the distribution of $x$ and that of $p$, but there is no joint probability distribution of $(x,p)$ with the correct marginals of $x$ and $p$. To overcome this shortcoming, Wigner came up with a quasiprobability distribution $W_{\psi}(x,p)$, that is, a countably additive function on the measurable subsets of state space $\Omega=\mathbb{R}^2$ such that $W_{\psi}(\Omega)=1$. Wigner's function $W_{\psi}(x,p)$ is a signed probability distribution; all values of  $W_{\psi}(x,p)$ are real, but some values may be negative. In \cite{blass}, the authors show that Wigner’s quasiprobability distribution is the unique signed probability distribution yielding the correct marginal distributions for position and momentum and all their linear combinations, and in \cite{vovk}, the authors show that Wigner's function can be tested using the canonical approach of  game-theoretic probability. We can give Wigner's function two interpretations. The first one is purely ``utilitarian'': because it is the unique quasiprobability distribution that gives the correct marginal distributions for $x$ and $p$, scientists should use it even if it allows for probabilities being negative-valued. The second one can be reconciled with the interpretation we give in the present work to negative extended probabilities. Because a scientist cannot observe $x$ and $p$ simultaneously, it is fair to think that they cannot enter a bet about events involving both position and momentum of a particle; hence, they resort to negative probabilities.

In \cite{benavoli2}, the authors study bounds on the algorithmic capabilities of a mathematical theory and analyze their implications. In particular, they require that the theory is logically consistent -- that is, it has to be based on a few axioms and rules from which mathematical truths can be unambiguously derived -- and that inferences in the theory should be computable in polynomial time -- that is, there should be an efficient way to execute the theory. The postulates of consistency and computation are apparently in conflict with each other: intuitively, if only polynomial time computations are allowed, the theory is consistent only up to what polynomial calculus allows, an instance that the authors call external-internal clash. They formalise such a clash via a ``weirdness theorem'', which shows that any theory obeying their two postulates necessarily departs in a very peculiar way from the probabilistic point of view. In particular, the theorem proves that all models compatible with the theory will present some negative probabilities (quasiprobabilities à la Wigner). They also show that quantum paradoxes are a consequence of the weirdness theorem. To develop their results, the authors rely on a characterization of probability in terms of lotteries (or gambles, a particular type of bets). They then provide
a subjective foundation, à la de Finetti, of so-called generalised probability theories. Negative probabilities arise when an agent would like to enter a given bet involving an event, but computational limitations related to the event prevent them from doing so. This interpretation coincides with the one we give in the present paper: the authors give a reason for which an agent is denied the possibility of entering a bet.

Finally, the interpretation of negative probabilities in \cite{gell-mann} and \cite{hartle1,hartle2} is very similar -- although not identical -- to  the one in the present work. We give the summary of their way of interpreting negative probabilities as reported in \cite{feintzeig}. First, they point out that the probabilities dictated by a physical theory instruct (rational) agents on how to bet on the outcomes of phenomena. Standard arguments from subjective Bayesian probability theory -- the Dutch book arguments -- demand that the Kolmogorovian axioms for classical probability theory must hold for the degrees of belief of any rational agent, which determine which bets the agent regard as fair. However, they point out that these arguments only apply to bets which are settleable, that is, bets about events the agent will certainly know at some point. They then argue that only bets on sufficiently coarse-grained alternative histories will be settleable, where this coarse-graining guarantees that the alternatives’ probabilities will lie in the unit interval. The main difference between their interpretation and the one presented in this paper is that for Gell-Mann and Hartle the bets on events that have negative probabilities cannot be settled, whereas we deem those bets to be settlable. The agent simply cannot enter them at the moment, but they may be given the opportunity in the future.

\section{Extended probabilities in statistical inference}\label{main}
In this section, we are going to consider an ex ante analysis in which we progressively learn the composition of the state space. It is ex ante in that it takes place before the actual statistical analysis. As time goes by, we collect new observations via a learning procedure specified in advance, e.g. an urn with or without replacement. 

Consider a measurable space $(\Omega,\mathcal{F})$ with $\Omega$ at most countable and $\mathcal{F}=2^\Omega$. At any time $t$, we have $\Omega=\Omega^-_t \sqcup \Omega^+_t$. For all $t \in \mathbb{N}_0:=\mathbb{N} \cup \{0\}$, $\Omega^-_t$ represents the ``latent'' part of $\Omega$ at time $t$, that is, the part that we have not yet observed. 
$\Omega^+_t$ represents the ``actual'' part of $\Omega$, that is, the portion of $\Omega$ that we have observed at time $t$ (at time $t=0$, $\Omega^+_0$ is the portion of $\Omega$ that we know ex ante, which is assumed nonempty). This approach is similar to the ``ex ante humility'' one introduced in \cite[section 1]{allen}, where latent and actual portions of the state space are first introduced, and a dynamic is described. A graphical representation of $\Omega=\Omega^-_t \sqcup \Omega^+_t$ is given in Figure \ref{fig3}. 
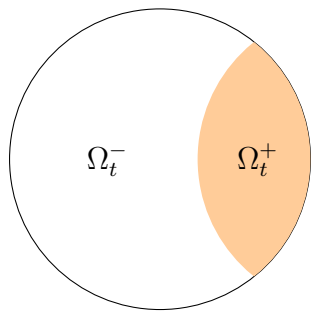
\begin{figure}[h!]
\centering
\begin{tikzpicture}
\draw[clip] (0,0) circle (2cm);
\fill[orange!40] (2.5,0) circle (2cm);
\node at (-.7,0) {$\Omega^-_t$};  
\node at (1.3,0) {$\Omega^+_t$};
\end{tikzpicture}
\caption{Graphical representation of $\Omega=\Omega^-_t \sqcup \Omega^+_t$.} 
\label{fig3}
\centering
\end{figure}
At time $t=0$, we specify the finest possible partition of $\Omega$, $\mathcal{E}=\{\{\omega\}\}_{\omega\in\Omega}$. It is such that $\mathcal{E}=\mathcal{E}_0^- \sqcup \mathcal{E}_0^+$; in this notation, $\mathcal{E}_0^-$ partitions $\Omega^-_0$, and $\mathcal{E}_0^+$ partitions $\Omega^+_0$. The most general way of proceeding is by specifying the number type (natural numbers $\mathbb{N}$, whole numbers $\mathbb{N}_0$, integers $\mathbb{Z}$, or rationals $\mathbb{Q}$) we subjectively believe we are working with, and setting it to be $\Omega$. To this extent, we give to $\Omega$ the {apparently possible} interpretation of  \cite[section 2.1.2]{walley}: $\Omega$ is the space of apparently possible states if it contains all the states $\omega$ that we believe are logically consistent with our available information. To give an example, a space of apparently possible states associated with a coin toss is 
\begin{align*}
\Omega=\{&\text{heads, tails, coin landing on its edge, coin breaking into pieces on landing,}\\ &\text{coin disappearing down a crack in the floor}\}.
\end{align*}
Cosidering the space of apparently possible states, then, amounts to considering the broadest state space associated with the statistical experiment of interest.

After that, we subjectively specify what we initially think the state space is, and we set it to be $\Omega_0^+$; for example, we may have previous information coming from similar (but not equal) experiments.         

Given the sequences $(\Omega_t^+)$ and $(\Omega_t^-)$, we require that, for all $t$,
$$\Omega^+_{t} \subset \Omega^+_{t+1} \quad \text{and} \quad \Omega^-_{t} \supset \Omega^-_{t+1}.$$
This means that $(\Omega_t^+)$ is monotone nondecreasing, and $(\Omega_t^-)$ is monotone nonincreasing, which implies that the limits of both exist and are well defined. In particular, we have that
$$ \lim_{t \rightarrow \infty} \Omega^+_t= \bigcup_{t \in \mathbb{N}_0} \Omega^+_t \quad \text{and} \quad \lim_{t \rightarrow \infty} \Omega^-_t= \bigcap_{t \in \mathbb{N}_0} \Omega^-_t.$$
We then have two possible scenarios. In the first one, $\bigcup_{t \in \mathbb{N}_0} \Omega^+_t=\Omega$, which implies that $\bigcap_{t \in \mathbb{N}_0} \Omega^-_t=\emptyset$. In the second one, $\bigcup_{t \in \mathbb{N}_0} \Omega^+_t \equiv \Omega^\prime \subsetneq \Omega$, which implies that $\bigcap_{t \in \mathbb{N}_0} \Omega^-_t \equiv \Omega^{\prime \prime}:=\Omega \setminus \Omega^\prime\neq\emptyset$.

As we collect more and more observations, we progressively discover the composition of our sample space. In the first scenario, in the limit we discover that the sample space associated with our experiment corresponds to the whole set $\Omega$ we initially specified. In this scenario, at time $t+1$ when an element $\tilde{\omega}$ is observed that belongs to $\Omega^-_t$, the element $E^-_{t,\tilde{\omega}}=\{\tilde{\omega}\}$ of the partition $\mathcal{E}^-_t$ consistent with $\tilde{\omega}$ becomes an element of $\mathcal{E}^+_{t+1}$. We use the following notation: $E^+_{t+1,\tilde{\omega}}=E^-_{t,\tilde{\omega}}$. So the partition of the latent space loses an element, while the partition of the actual space gains one. This means that, for all $t$, $\mathcal{E}^+_{t} \subset \mathcal{E}^+_{t+1}$ and $\mathcal{E}^-_{t} \supset \mathcal{E}^-_{t+1}$. We abuse notation: we write $\mathcal{E}^+_{t} \subset \mathcal{E}^+_{t+1}$ to indicate that $\#\mathcal{E}^+_{t} \leq \#\mathcal{E}^+_{t+1}$, where $\#\mathcal{E}^+_{t}$ denotes the number of elements of the partition $\mathcal{E}^+_{t}$, for all $t$, and we write $\mathcal{E}^-_{t} \supset \mathcal{E}^-_{t+1}$ to indicate that $\#\mathcal{E}^-_{t} \geq \#\mathcal{E}^-_{t+1}$. Hence, in the limit, the partition $\mathcal{E}_t^+$ of the ``actual'' space coincides with the partition $\mathcal{E}$ of the whole set $\Omega$. This mirrors the behavior of $\Omega_t^+$, that converges to $\Omega$.

In the second scenario $\Omega^\prime$ can be finite or countable. In the former case, there exists a $T \in \mathbb{N}$ after which the observations we collect at time $T+i$ already belong to $\Omega_T^+$, for all $i \in \mathbb{N}$. If that is the case, we write
\begin{align}\label{eq3}
\Omega^\prime \equiv \Omega^+_T.
\end{align}
This corresponds to discovering that the actual sample space is finite and smaller than the whole set $\Omega$ that we specified at the beginning of our analysis. If $\Omega^\prime$ is countable, we can only say that it is a proper subset of $\Omega$, and that we discover its composition in the limit. This may happen, for example, if we begin our analysis by setting $\Omega=\mathbb{N}_0$, but then realize that the state space associated with our experiment is actually $\Omega^\prime=\mathbb{N}$.

In general, when the second scenario takes place, we have $\Omega_t^+ \uparrow \Omega^\prime$, and $\Omega^-_t \downarrow \Omega^{\prime\prime}:=\Omega \setminus \Omega^\prime$. We also have that, in the limit, the partition $\mathcal{E}_t^+$ of the ``actual'' space coincides with the (finest possible) partition $\mathcal{E}^\prime$ of set $\Omega^\prime$. 
We call $\mathcal{E}^{\prime \prime}$ the partition whose elements belong to $\mathcal{E}$ but not to $\mathcal{E}^\prime$; again abusing notation, we write $\mathcal{E}^{\prime\prime}:=\mathcal{E} \setminus \mathcal{E}^\prime$. We call $\mathcal{F}^\prime=2^{\Omega^\prime}$ and $\mathcal{F}^{\prime\prime}=2^{\Omega^{\prime\prime}}$.

\begin{remark}\label{benevolent}
In this paper, our ex ante analysis is concluded by finding the true composition of the state space. We can interpret this using the concept of a benevolent bookmaker. While our doppelgänger is always able to enter a bet on all the events in the space $\Omega$ of apparently possible states, we end up only able to bet on the events of $\Omega$ that are crucial to the statistical analysis taking place after our ex ante analysis. This is akin to a benevolent bookmaker preventing us from entering bets on irrelevant events, that, if we were to bet on, would certainly make us lose money.\footnote{A more correct expression would be ``almost benevolent'' since we are considering only one-way bets.}

However, there may be cases in which the state space does not get fully discovered. For example, this may happen if we have an urn with a timer attached; once the time runs out,  the urn is sealed, so that we do not discover its entire composition. The following statistical analysis, which would require the use of extended probabilities, is explored in a future work. This case corresponds to the existence of a malevolent bookmaker that does not allow us to enter bets that are crucial to the statistical analysis.\footnote{Here ``malevolence'' has to be understood in terms of restricting the menu of available options to the agent.} This setting is especially important when studying events about which we will never be able to enter a bet, for example latent events.
\end{remark}

\begin{remark}\label{rem_on_omega}
If at any time $t$ we collect an observation $\breve{\omega}$ that does not belong to $\Omega$, this means that the space $\Omega$ we initially specified is not rich enough. We have then to specify a richer, larger set $\breve{\Omega} \supset \Omega$, and start our analysis over. We can either consider $\breve{\Omega}=\Omega \cup \{\breve{\omega}\}$, or define $\breve{\Omega}$ as a larger number type. Let us give an example. Suppose we begin our analysis by setting $\Omega=\mathbb{N}$, and after a while we observe $\breve{\omega}=1/2$. Then, we have to restart our analysis and we can either let $\breve{\Omega}$ be $\mathbb{N}\cup\{1/2\}$ or $\breve{\Omega}=\mathbb{Q}$.

It may also happen that our true sample space is $\breve{\Omega} \subsetneq \Omega_0^+$. In this case, equation (\ref{eq3}) holds with $T=0$, and our analysis would still be valid. The drawback is that, if that happens, we are not respecting one of the conditions listed in \cite{tsitsiklis} for a sample space to be valid. In particular, that the sample space $\Omega$ must have the right granularity depending on what we are interested in. This means that we must remove irrelevant information from the sample space. In other words, we must choose the right abstraction and forget irrelevant information.
This issue can be avoided by initially specifying $\Omega_0^+$ so that it has the fewest possible elements.
\end{remark}

On top of dealing with uncertainty on the composition of the sample space, we also address the problem of not being able to specify a unique (extended) probability measure on $\Omega$. That is why we are going to work with sets of extended probability measures. We call this approach \textit{extended sensitivity analysis}, since it corresponds to the extended probabilities counterpart of Bayesian sensitivity analysis \cite{berger}. We begin by considering a set $\mathcal{P}^{ex} \equiv \mathcal{P}^{ex}_0$ of extended probabilities that represent the agent's initial beliefs, and we update it as described in section \ref{update}. We denote the sequence of successive updates of $\mathcal{P}^{ex}_0$ as $(\mathcal{P}^{ex}_t)_{t \in \mathbb{N}}$. By working with sets of extended probability measures, we represent the condition of a researcher facing a decision under ambiguity \cite{ellsberg}. As we shall see in section \ref{update}, sequence $(\mathcal{P}^{ex}_t)_{t \in \mathbb{N}_0}$ converges in the Hausdorff metric.

Now, fix any $t \in \mathbb{N}_0$, and consider any $P^{ex}_t \in \mathcal{P}^{ex}_t$. We require the following
\begin{itemize}
\item[(i)] $P^{ex}_t(A) \in [0,1]$ if $A \subset \Omega_t^+$;
\item[(ii)] $P^{ex}_t(A) \in [-1,0]$ if $A \subset \Omega_t^-$;
\item[(iii)] $P^{ex}_t(A) \in [-1,1]$ if $A \cap \Omega_t^+ \neq \emptyset \neq A \cap \Omega_t^-$.
\end{itemize}
In particular, we compute the latter as follows.
\begin{proposition}\label{prop4}
The following is true.
\begin{align}\label{eq14}
\begin{split}
P^{ex}_t(A) &=\sum\limits_{E_{t,j}^+ \in \mathcal{E}_t^+ : P^{ex}_t (E_{t,j}^+)\neq 0} P^{ex}_t ( A \mid E_{t,j}^+ )P^{ex}_t(E_{t,j}^+)\\ &+ \sum\limits_{E_{t,j}^- \in \mathcal{E}_t^- : P^{ex}_t (E_{t,j}^-)\neq 0} P^{ex}_t ( A \mid E_{t,j}^- )P^{ex}_t(E_{t,j}^-).
\end{split}
\end{align}
\end{proposition}
Notice that we need to consider the elements of $\mathcal{E}_t^+$ and $\mathcal{E}_t^-$ whose extended probabilities are not $0$ otherwise the conditional extended probabilities  $P^{ex}_t ( A \mid E_{t,j}^+ )$ and $P^{ex}_t ( A \mid E_{t,j}^-)$ may result in an indeterminate form of the $\frac{0}{0}$ kind. This requirement yields no loss of generality since if an element of a partition is assigned extended probability $0$, then it does not convey any information around event $A$. Notice also that, for the elements of $\mathcal{E}_t^+$ and $\mathcal{E}_t^-$ whose extended probabilities are not $0$, $P^{ex}_t ( A \mid E_{t,j}^+ )=\frac{P^{ex}_t ( A \cap E_{t,j}^+ )}{P^{ex}_t ( E_{t,j}^+ )} \geq 0$ by (i), and $P^{ex}_t ( A \mid E_{t,j}^- )=\frac{P^{ex}_t ( A \cap E_{t,j}^- )}{P^{ex}_t ( E_{t,j}^- )} \geq 0$ because it is the ratio of two negative quantities; once multiplied by $P^{ex}_t(E_{t,j}^-)$, which is negative by (ii), it gives us a negative value. So the sign of $P^{ex}_t(A)$ is not predetermined when $A \cap \Omega_t^+ \neq \emptyset \neq A \cap \Omega_t^-$. The interpretation of these conditions is straightforward: we assign negative extended probabilities to events that belong to the latent space at time $t$ (meaning that at time $t$ we cannot enter a bet about them), while we assign positive extended probabilities to events that are in the actual, observed space (meaning that at time $t$ we can enter a bet about them). If a given event is only partially known, its extended probability has not a predetermined sign: it will depend on whether we know enough about it (then the probability will be positive), or not (vice versa). This means that we can enter a bet about ``sub-event'' $A \cap \Omega_t^+$, but not about $A \cap \Omega_t^-$. For example, let $A=\{\text{tomorrow there will be a thunderstorm}\}$. Then, for some $t$, suppose that
\begin{align*}
A \cap \Omega_t^+&=\{\text{tomorrow will rain}\},\\
A \cap \Omega_t^-&=\{\text{tomorrow there will be a dry thunderstorm}\}.
\end{align*}
Then we can place a bet on $A \cap \Omega_t^+$ but not on $A \cap \Omega_t^-$, so the extended probability we assign to $A$ does not have a predetermined sign.

Notice also that we can define a set of events $\mathscr{C}_{P^{ex}_t}:=\{ A \in \mathcal{F} : P^{ex}_t(A \cap \Omega^-_t)=-P^{ex}_t(A \cap \Omega^+_t) \}$, which we call {critical events according to $P^{ex}_t$}, with the property that $P^{ex}_t(A)=0$, for all $A \in \mathscr{C}_{P^{ex}_t}$ (immediate from the definition). We assign the sub-event we can enter a bet about the same probability that our doppelgänger assigns to the sub-event that we are not allowed to bet on. This means that we deem their ``actual'' portion (the one we know/we have observed so far) to be just as likely than their ``latent'' portion (the one we do not know/we have not yet observed). The set of critical events according to the whole set of extended probability measures $\mathcal{P}^{ex}_t$ is given by
$$\mathscr{C}_{t}:= \bigcap\limits_{P^{ex}_t \in \mathcal{P}^{ex}_t}\mathscr{C}_{P^{ex}_t}.$$
\subsection{Properties of this environment}
We now give some results concerning the environment we depicted so far. 

\begin{proposition}\label{prop7}
Let $A^+_t:=A \cap \Omega_t^+$, $A^-_t:=A \cap \Omega_t^-$, $B^+_t:=B \cap \Omega_t^+$, and $B^-_t:=B \cap \Omega_t^-$. Then, $A \cup B=(A^+_t \cup B^+_t)\cup (A^-_t \cup B^-_t)$ and $A \cap B=(A^+_t \cap B^+_t)\cup (A^-_t \cap B^-_t)$, for all $t$. Also, $A \setminus B=(A^+_t \setminus B^+_t) \cup (A^-_t \setminus B^-_t)$, for all $t$.
\end{proposition}

Recall now that a set ring is a system of sets $\mathbb{B}$ such that $A,B \in \mathbb{B}$ implies $A\cap B \in \mathbb{B}$ and $(A\setminus B) \cup (B\setminus A) =: A \triangle B\in \mathbb{B}$. A set ring $\mathbb{B}$ with a unit element, i.e. $E \in \mathbb{B}$ such that for all $A \in \mathbb{B}$, $A \cap E=A$, is called a set algebra. Let $\mathcal{F}^+_t=2^{\Omega^+_t}$ and $\mathcal{F}^-_t=2^{\Omega^-_t}$.
\begin{proposition}\label{prop8}
$\mathcal{F}_t^+$ and $\mathcal{F}_t^-$ are set algebras, for all $t$.
\end{proposition}

We also point out that if $A=\{\omega_1,\ldots,\omega_k\}$ and $\omega_1,\ldots,\omega_k \in  \Omega$, then 
$$P^{ex}(A)=\sum_{j=1}^k P^{ex}(\{\omega_j\}).$$
This is immediate from the countable additivity of extended probabilities.

Another property is the following. Fix any $t$ and let $A,B \in\mathcal{F}_t^+$ such that $A \subset B$. Then, $P^{ex}_t(A) \leq P^{ex}_t(B)$. Let then $C,D \in\mathcal{F}_t^-$ such that $C \subset D$. Then, $P^{ex}_t(C) \geq P^{ex}_t(D)$. Both these results come from Proposition \ref{prop2}.

The following is also interesting.
\begin{proposition}\label{prop9}
Consider  $A \subset \cup_{j \in \mathbb{N}_0} A_j$, where $A_j \in \mathcal{F}$ for all $j$, and also $A \in \mathcal{F}$. Then $P^{ex}_t(A) \leq \sum_{j \in \mathbb{N}_0} P^{ex}_t(A_j)$ if $\cup_{j \in \mathbb{N}_0} A_j \in\mathcal{F}_t^+$, and $P^{ex}_t(A) \geq \sum_{j \in \mathbb{N}_0} P^{ex}_t(A_j)$ if $\cup_{j \in \mathbb{N}_0} A_j \in\mathcal{F}_t^-$. If instead some of the $A_j$'s are in $\mathcal{F}_t^+$ and some are in $\mathcal{F}_t^-$, then $P^{ex}_t(A) \gtreqless \sum_{j \in \mathbb{N}_0} P^{ex}(A_j)$, that is, $P^{ex}_t(A)$ can be larger, smaller, or equal to $\sum_{j \in \mathbb{N}_0} P^{ex}(A_j)$.
\end{proposition}

In the setting we have outlined so far, there is a way of operationalizing equation (\ref{eq1}). Pick any $A \in \mathcal{F}$; we have
\begin{align}\label{eq2}
P^{ex}_t(A^c)=P^{ex}_t \left( \Omega_t^+ \setminus [\Omega^+_t \cap A] \right) + P^{ex}_t \left( \Omega_t^- \setminus [\Omega^-_t \cap A] \right)
\end{align}
We retain the fact that $P^{ex}_t(\emptyset)=0$; indeed
\begin{align*}
P^{ex}_t(\Omega^c)&=P^{ex}_t(\emptyset)=P^{ex}_t \left( \Omega_t^+ \setminus [\Omega^+_t \cap \Omega] \right) + P^{ex}_t \left( \Omega_t^- \setminus [\Omega^-_t \cap \Omega] \right)\\
&=P^{ex}_t(\emptyset)+P^{ex}_t(\emptyset)\\
&\iff P^{ex}_t(\emptyset)=0.
\end{align*}
We can also write $P^{ex}_t(A^c)=P^{ex}_t(A^c \cap \Omega^+_t)+P^{ex}_t(A^c \cap \Omega^-_t)$, because, as we know, $A^c=(A^c \cap \Omega^+_t) \sqcup (A^c \cap \Omega^-_t)$.
\subsection{Interpretation and updating procedure}\label{update}
Let us now discuss the probability assigned to the whole sample space $\Omega$. From (ii*), we know that since, for all $t$, $\Omega_t^+ \sqcup \Omega_t^-=\Omega$, then $P^{ex}_t(\Omega)=P^{ex}_t(\Omega_t^+)+P^{ex}_t(\Omega_t^-)$. Also, from (i) and (ii) we know that $P^{ex}_t(\Omega_t^+) \geq 0$ and $P^{ex}_t(\Omega_t^-) \leq 0$. So
\begin{align*}
P^{ex}_t(\Omega)=
\begin{cases}
      p>0 & \text{ if } P^{ex}_t(\Omega_t^+)>|P^{ex}_t(\Omega_t^-)|\\
      p=0 & \text{ if } P^{ex}_t(\Omega_t^+)=|P^{ex}_t(\Omega_t^-)|\\
      p<0 & \text{ if } P^{ex}_t(\Omega_t^+)<|P^{ex}_t(\Omega_t^-)|
    \end{cases}.
\end{align*}
What does it mean, then, for $P^{ex}_t(\Omega)$ to be equal to $0$? And to be negative? And to be positive, but not $1$? Given the benevolent bookmaker interpretation, we have the following. If $P^{ex}_t(\Omega)=0$, it means that we have no sufficient information to say whether the sample space associated with our experiment is in fact $\Omega$. If $P^{ex}_t(\Omega)<0$, it means that, for the time being, the sample space associated with our experiment appears to be some $\breve{\Omega} \subsetneq \Omega$. If $P^{ex}_t(\Omega) \in (0,1)$, it means that there is evidence that the sample space associated with our experiment could be in fact $\Omega$, but we cannot state it with certainty. 

The natural question that one might ask now is, for some $t \in \mathbb{N}_0$, how do we come up with negative numbers to assign to events that belong to the latent space $\Omega^-_t$. Or, for that matter, how do we come up with positive numbers to assign to events that belong to the actual space $\Omega^+_t$. Call $\Delta(\Omega,\mathcal{F})$ the set of probability measures on $(\Omega,\mathcal{F})$, and $\Delta^{ex}(\Omega,\mathcal{F})$ the set of extended probability measures on $(\Omega,\mathcal{F})$. The latter is a linear space, as shown in \cite{rao}. In a future work, we will argue that it is a Dedekind complete Banach lattice with respect to the norm induced by the total variation of an element $P^{ex}$ of $\Delta^{ex}(\Omega,\mathcal{F})$. We will also show that if $\Omega$ is a compact separable space, then the subset $\Delta^{ex}_{Baire}(\Omega,\mathcal{F}) \subset \Delta^{ex}(\Omega,\mathcal{F})$ of Baire extended probability measures is the dual of the real Banach space of all continuous real-valued functions on $\Omega$.

Consider any $P \in \Delta(\Omega,\mathcal{F})$ such that, for all $\tilde{A} \in \mathcal{F}^+_0$, $P(\tilde{A})=p\in [0,1]$ is the amount we deem fair to pay to enter a bet about $\tilde{A}$; we call it the \textit{oracle probability measure}.
Then, for a generic $A \in \mathcal{F}$ we have that
\begin{equation}\label{guidance_p}
P^{ex}_0(A\cap\Omega^+_0)=P(A\cap\Omega^+_0)
\end{equation}
and
\begin{equation}\label{guidance_p2}
P^{ex}_0(A\cap\Omega^-_0)=-P(A\cap\Omega^-_0).
\end{equation}

It is easy to see that $P^{ex}_0$ satisfies (i*) and (ii*). Notice also that
\begin{align}\label{eq35}
\sum_{E \in \mathcal{E}} \left| P^{ex}_0(E) \right|=\sum_{E_0^+ \in \mathcal{E}_0^+} P^{ex}_0(E_0^+) + \sum_{E_0^- \in \mathcal{E}_0^-} \left| P^{ex}_0(E_0^-) \right|=1,
\end{align}
so $P^{ex}_0$ satisfies (iii*) as well. It is hence a properly defined extended probability measure. Clearly, \eqref{eq35} holds for all $t \in \mathbb{N}_0$, not just for $t=0$.

This way of assessing initial extended probabilites well reconciles with the interpretation we gave in general for extended probabilities: we cannot enter bets about events $A \not\in \mathcal{F}^+_0$ (in this case, because we cannot observe them for the time being). So we assess the probabilites of the events $A \in \mathcal{F}^+_0$ as specified in section \ref{extended}, and then we ask our doppelgänger the probabilities $P_{0,{D}}(B) \geq 0$ they assign to the elements $B \in \mathcal{F}^-_0$.\footnote{Subscript $D$ stands for ``doppelgänger''.} Then, we flip the sign to those, that is, $P^{ex}_0(B)=-P_{0,D}(B)$, for all $B \in \mathcal{F}^-_0$. In this notation, $P_{0,D}(B)$ is the probability the doppelgänger assigns to event $B$ at time $t=0$.

Because the agent faces ambiguity, they need to specify a set $\mathcal{P} \subset \Delta(\Omega,\mathcal{F})$ of probability measures. Every element $P \in \mathcal{P}$ induces an extended probability $P_0^{ex}$ as we just described. In this way, we build the set $\mathcal{P}_0^{ex}$.

Let us now discuss  how to update extended probabilities, that is, how to update $P^{ex}_t(A)$ to $P^{ex}_{t+1}(A)$, for all $A \in \mathcal{F}$, for all $P^{ex}_t \in \mathcal{P}^{ex}_t$, for all $t \in \mathbb{N}_0$.\footnote{It is worth mentioning that in \cite{walley_bom} too the author discusses a process of learning the state space combined with learning about the probability of the elements of $\mathcal{F}$.} We first consider a procedure to discover the components of the sample space that is equivalent to an urn without replacement. That is, after specifying $\Omega$ and $\Omega_0^+$, we start our analysis with an urn whose content is unknown and possibly countable; it represents the true sample space associated with our experiment. At any time point $t$, we extract a ball (an element $\omega$ of the sample space). Once we learn about that element, our knowledge about the composition of the urn increases. We do not put the ball back into the urn; we discuss the case in which the discover procedure is equivalent to an urn with replacement later in the paper.

Let us begin with the updating procedure from $P_0^{ex} \in \mathcal{P}_0^{ex}$ to $P_1^{ex}\in \mathcal{P}_1^{ex}$. 

We collect a new observation $\omega \in \Omega$. If $\omega \in\Omega_0^-$, this means that, at time $t=1$, we learn a new element of the true sample space. Notice that $\omega$ is consistent with an element $E^-_{0,\omega}$ of $\mathcal{E}_0^-$ (that is, $E^-_{0,\omega}=\{\omega\}$), which -- given that we observe such $\omega$ -- at time $t=1$ becomes an element of $\mathcal{E}_1^+$; in formulas, $E^-_{0,\omega}=E^+_{1,\omega}$. Then, $\mathcal{E}_1^+ \supset \mathcal{E}_0^+$, $\mathcal{E}_1^- \subset \mathcal{E}_0^-$, and we update the extended probability assigned to $E^-_{0,\omega}=E^+_{1,\omega}$ as follows
\begin{align}\label{eq4_1}
P^{ex}_1(E^+_{1,\omega})=\left| P^{ex}_0(E^-_{0,\omega}) \right|.
\end{align}
The extended probabilities assigned to the other elements of $\mathcal{E}$ (the finest possible partition of the whole $\Omega$) are held constant. From \eqref{eq14}, the updated extended probability associated with event $A$ is given by
\begin{align}\label{eq5_1}
\begin{split}
P^{ex}_1(A)=&\sum\limits_{E^+_{1,j} \in \mathcal{E}^+_1 : P^{ex}_1(E^+_{1,j}) \neq 0} P^{ex}_1(A \mid E^+_{1,j}) P^{ex}_1(E^+_{1,j})\\ &+ \sum\limits_{E^-_{1,j} \in \mathcal{E}^-_1 : P^{ex}_1(E^-_{1,j}) \neq 0} P^{ex}_1(A \mid E^-_{1,j}) P^{ex}_1(E^-_{1,j}),
\end{split}
\end{align}
where
$$P^{ex}_1(A \mid E^+_{1,j})=\frac{P^{ex}_1(A \cap E^+_{1,j})}{P^{ex}_1(E^+_{1,j})}$$and$$P^{ex}_1(A \mid E^-_{1,j})=\frac{P^{ex}_1(A \cap E^-_{1,j})}{P^{ex}_1(E^-_{1,j})}.$$
Notice that we are working with the finest possible partition of $\Omega$, that is, $\mathcal{E}=\{\{\omega\}\}_{\omega\in \Omega}$. Then, for any $A \in \mathcal{F}$ and any $E \in\mathcal{E}$, $A \cap E=A \cap \{\omega\}$, which is equal to the empty set if $\omega \not\in A$, and it is equal to $\{\omega\}=E$ if $\omega \in A$. Hence, for all $E \in \mathcal{E}$ such that $P_1^{ex}(E)\neq 0$,
$$P^{ex}_1(A \mid E)=\frac{P^{ex}_1(A \cap E)}{P^{ex}_1(E)}=\begin{cases}
\frac{P^{ex}_1(\emptyset)}{P^{ex}_1(E)}=0 & \text{if } A \cap E= \emptyset\\
\frac{P^{ex}_1(E)}{P^{ex}_1(E)}=1 & \text{if } A \cap E= \{\omega\}
\end{cases},$$
for all $E \in \mathcal{E}$.\\
So, equation \eqref{eq5_1} can be rewritten as
\begin{align}\label{eq40}
\begin{split}
P^{ex}_1(A)=&\sum\limits_{E^+_{1,j} \in \mathcal{E}^+_1: A \cap E^+_{1,j} \neq \emptyset \text{, } P^{ex}_1(E^+_{1,j}) \neq 0} P^{ex}_1(E^+_{1,j})\\ 
&+ \sum\limits_{E^-_{1,j} \in \mathcal{E}^-_1: A \cap E^-_{1,j} \neq \emptyset \text{,  }P^{ex}_1(E^-_{1,j}) \neq 0} P^{ex}_1(E^-_{1,j}).
\end{split}
\end{align}

Of course this holds for all $P^{ex}_1 \in \mathcal{P}^{ex}_1$. Notice that we are implicitly assuming Allen's principle of conservation of knowledge. Indeed, suppose at time $\mathbf{t}$ the event $A$ is entirely in the latent portion of $\Omega$, that is, $A \cap \Omega_\mathbf{t}^-=A$, and at time $\mathbf{t}+k$ it is entirely in the actual portion of $\Omega$, that is, $A \cap \Omega_{\mathbf{t}+k}^+=A$. Then, by \eqref{eq4_1} and \eqref{eq40}, we have that $P^{ex}_{\mathbf{t}+k}(A)=|P^{ex}_{\mathbf{t}}(A)|$. This has the practical advantage of making the updating procedure mechanical (we do not need to reassess any subjective extended probability when new observations become available), and of preserving the initial opinion of the researcher. The importance of this last point is discussed in Remark \ref{rem21}.

If instead we observe $\omega \in\Omega_0^+$, this means that we made a good job in specifying $\Omega_0^+$, and so we keep the extended probability constant
\begin{align}\label{eq41}
P^{ex}_1(E^+_{1,\omega})=P^{ex}_0(E^+_{0,\omega}) \geq 0.
\end{align}
Of course, the extended probabilities assigned to the other elements of $\mathcal{E}$ remain constant as well. 

We follow the procedures in \eqref{eq4_1} and \eqref{eq41} to update $P^{ex}_t$ to $P^{ex}_{t+1}$ for all $t \in \mathbb{N}_0$, not just from $t=0$ to $t=1$. 
We call $(P^{ex}_t)$ the sequence of updates of initial extended probability $P^{ex}_0$. Now consider the extended probability $P^{ex}_\infty \in \Delta^{ex}(\Omega,\mathcal{F})$ such that
\begin{equation*}
P^{ex}_\infty \left( A\cap \bigcup_{t \in \mathbb{N}_0}\Omega^+_t \right)=P \left( A\cap \bigcup_{t \in \mathbb{N}_0}\Omega^+_t \right)
\end{equation*}
and
\begin{equation*}
P^{ex}_\infty \left( A\cap \bigcap_{t \in \mathbb{N}_0}\Omega^-_t \right)=-P \left( A\cap \bigcap_{t \in \mathbb{N}_0}\Omega^-_t \right),
\end{equation*}
for all $A \in \mathcal{F}$, where $P$ is the oracle probability measure we used in \eqref{guidance_p} and \eqref{guidance_p2} to specify $P^{ex}_0$. Let us denote by $d_{ETV}$ the extended total variation distance, 
$$d_{ETV}(P^{ex},Q^{ex}):=\sup_{A \in \mathcal{F}} \left| P^{ex}(A)-Q^{ex}(A) \right|.$$
It is routine to check that $d_{ETV}$ is a metric: the proof goes along the lines of showing that the total variation distance is a metric. Then, the following holds.
\begin{proposition}\label{prop11}
${P}^{ex}_t \rightarrow {P}^{ex}_\infty$ as $t \rightarrow \infty$ in the extended total variation metric.
\end{proposition}
The following claim is especially important.
\begin{proposition}\label{prop10}
If $\Omega^+_t \uparrow \Omega$, then $P^{ex}_\infty(\Omega)=1$.
\end{proposition}
This result implies that if $\Omega^+_t \uparrow \Omega$, then $P^{ex}_\infty$ is a regular probability measure. Indeed, it is easy to see that it satisfies the (countably additive) Kolmogorovian axioms for regular probability measures.


Now call $\mathcal{P}^{ex}_\infty \subset \Delta^{ex}(\Omega,\mathcal{F})$ the following set
$$\mathcal{P}^{ex}_\infty:=\left\lbrace{P^{ex}_\infty \in \Delta^{ex}(\Omega,\mathcal{F}) : d_{ETV}(P^{ex}_t\text,P^{ex}_\infty)\xrightarrow[t \rightarrow \infty]{}0 \text{, } P^{ex}_t \in \mathcal{P}^{ex}_t}\right\rbrace.$$
That is, $\mathcal{P}^{ex}_\infty$ is the set of limits (in the extended total variation metric) of the sequences $(P^{ex}_t)$ whose elements $P^{ex}_t$ belong to $\mathcal{P}^{ex}_t$. 

The Hausdorff distance between an element of $(\mathcal{P}^{ex}_t)_{t \in \mathbb{N}_0}$ and $\mathcal{P}^{ex}_\infty$ is given by
\begin{align}\label{eq26}
\begin{split}
{d}_H\left( \mathcal{P}^{ex}_{t},\mathcal{P}^{ex}_\infty \right)= \max \bigg\{ &\sup_{P^{ex}_t \in \mathcal{P}^{ex}_t} \inf_{P^{ex}_\infty \in \mathcal{P}^{ex}_\infty}  d_{ETV}\left( P_t^{ex},P_\infty^{ex} \right), \\& \sup_{P^{ex}_\infty \in \mathcal{P}^{ex}_\infty} \inf_{P^{ex}_t \in \mathcal{P}^{ex}_t} d_{ETV}\left( P_t^{ex},P_\infty^{ex} \right) \bigg\}.
\end{split}
\end{align}
Then, the sequence $(\mathcal{P}^{ex}_t)_{t \in \mathbb{N}_0}$ of successive updates of set $\mathcal{P}^{ex}_0$ representing the initial beliefs of the agent facing ambiguity converges in the Hausdorff distance to $\mathcal{P}^{ex}_\infty$. This result is an immediate consequence of Proposition \ref{prop11}: every element of $\mathcal{P}^{ex}_t$ converges to an element of $\mathcal{P}^{ex}_\infty$, so the distance between the ``borders'' of these two sets -- measured by the Hausdorff metric -- converges to $0$.

Now, there are two possible scenarios: one in which we continue discovering the elements of the state space until we retrieve the full $\Omega$ we specified ex ante (this corresponds to $\Omega_t^+ \uparrow \Omega$), and another one in which we discover that the actual sample space associated with our experiment is $\Omega^\prime \subsetneq \Omega$. 
In the first scenario, any $P^{ex}_\infty \in \mathcal{P}^{ex}_\infty$ is a regular probability measure, a consequence of Proposition \ref{prop10}. So, after discovering the composition of the state space, we have that $(\Omega,\mathcal{F},{P}^{ex}_\infty)$ is a regular probability space, for all ${P}^{ex}_\infty \in \mathcal{P}^{ex}_\infty$.


In the second scenario, $\Omega^\prime$ can be finite or countable. In the former case, we have that, for some $T \in \mathbb{N}$, $\Omega_T^+ \equiv \Omega^\prime \subsetneq \Omega$, so $\sum_{E_T^+ \in \mathcal{E}^+_T} P^{ex}_T(E_T^+)=q<1$ because, by \eqref{eq35} and \eqref{eq4_1},
$$\sum_{E_T^+ \in \mathcal{E}^+_T} P^{ex}_T(E_T^+) + \sum_{E_T^- \in \mathcal{E}^-_T} \left| P^{ex}_T(E_T^-) \right|=1.$$
This may seem problematic: $\mathcal{P}^{ex}_t$ converges to $\mathcal{P}^{ex}_T$ (that is, $\mathcal{P}^{ex}_\infty$ coincides with $\mathcal{P}^{ex}_T$), which  is not a set of regular probability measures. To solve this issue, we need to describe the regular probability measure induced by every $P^{ex}_T \in \mathcal{P}^{ex}_T$. It would be desirable to find $\tilde{P}$ such that $\tilde{P}(A)=cP^{ex}_T(A)$, for all $A \in \mathcal{F}^+_T$. This because such a regular probability measure preserves the ratios between extended probabilities of the elements of $\mathcal{F}^+_T$, that is,
$$\frac{\tilde{P}(A)}{\tilde{P}(B)}=\frac{cP^{ex}_T(A)}{cP^{ex}_T(B)}=\frac{P^{ex}_T(A)}{P^{ex}_T(B)},$$
for all $A,B \in \mathcal{F}^+_T$ such that $P^{ex}_T(B) \neq 0$.\footnote{Notice that if the ratio is greater than $1$, we deem $A$ more likely than $B$, and vice versa if the ratio is smaller than $1$. We deem $A$ and $B$ equally likely if the ratio is exactly $1$. Because we want to preserve this behavioral interpretation when we move to the regular probabilities induced by the elements of $\mathcal{P}^{ex}_T$, we want the ratio to be maintained.} To find such a $c$, the following needs to hold
$$\sum_{E_T^+ \in \mathcal{E}^+_T} cP^{ex}_T(E_T^+)=1,$$
which happens if and only if
$$c=\frac{1}{\sum_{E_T^+ \in \mathcal{E}^+_T} P^{ex}_T(E_T^+)}=\frac{1}{P^{ex}_T(\Omega_T^+)}.$$
Hence, $\tilde{P}=cP^{ex}_T$ is the regular probability measure induced by $P^{ex}_T$ that preserves the ratios between extended probabilities of elements of $\mathcal{F}^+_T$. Clearly, this holds for all $P^{ex}_T \in \mathcal{P}^{ex}_T$, so $(\Omega^+_T, \mathcal{F}^+_T, \tilde{P})$ is a regular probability space, for all $\tilde{P} \in \tilde{\mathcal{P}}$, where $\tilde{\mathcal{P}}$ is the set of regular probabilities induced by the elements of $\mathcal{P}^{ex}_T$. 

If $\Omega^\prime$ is countable, we proceed in a similar way. We still want to preserve the ratios between extended probabilities of the elements of $\mathcal{F}^\prime$, so we have to find $c$ such that $\sum_{E \in \mathcal{E}^\prime}cP^{ex}_\infty(E)=1$. This happens when 
$$c=\frac{1}{\sum_{E \in \mathcal{E}^\prime}P^{ex}_\infty(E)}=\frac{1}{P^{ex}_\infty(\Omega^\prime)},$$
so $\tilde{P}=cP^{ex}_\infty$ is the regular probability measure we were looking for. This holds for all $P^{ex}_\infty \in \mathcal{P}^{ex}_\infty$, so $(\Omega^\prime,\mathcal{F}^\prime,\tilde{P})$ is a regular probability space, for all $\tilde{P} \in \tilde{\mathcal{P}}$.

\begin{remark}\label{rem21}
Notice that, in the first scenario ($\Omega_t^+ \uparrow \Omega$), every element $P^{ex}_\infty \in \mathcal{P}^{ex}_\infty$ coincides with its corresponding oracle probability measure $P \in \mathcal{P}$ we expressed ex ante on the whole $\Omega$. Indeed, we have that $P^{ex}_\infty(E)=|P^{ex}_\infty(E)|$, for all $E \in \mathcal{E}$, which implies $P^{ex}_\infty(A)=|P^{ex}_\infty(A)|$, for all $A \in\mathcal{F}$. But then, by \eqref{eq35}, we know that $\sum_{E \in \mathcal{E}} P^{ex}_\infty(E)=1$, and that $P^{ex}_\infty(A)=P(A)$, for all $A \in\mathcal{F}$. This should not surprise: the updating procedure we described earlier is not based on collecting new data like the Bayesian one, nor on repeating the experiment many times like the frequentist one. Rather, it is based on discovering the true composition of the state space associated with our inference procedure. Then, it is natural to retrieve the opinion we expressed ex ante on the state space $\Omega$ once we get the confirmation that the true state space is indeed $\Omega$. This also reconciles well with the interpretation we gave to negative extended probabilities: once we are given the possibility to enter the bets we were denied before, we tend to agree with our doppelgänger who had the opportunity to bet on those events in the first place (Allen's principle of conservation of knowledge holds).

Notice also that $\tilde{P}$, albeit not equal, is proportional to $P^{ex}_\infty$, $\tilde{P}=cP^{ex}_\infty$, for every $\tilde{P} \in \tilde{\mathcal{P}}$ and its corresponding $P^{ex}_\infty \in \mathcal{P}^{ex}_\infty$. The interpretation is immediate: we maintain our opinion on the elements of the sample space $\Omega^\prime \subsetneq \Omega$ that pertains to our experiment.
\end{remark}
If the procedure for discovering the composition of the state space is equivalent to an urn with replacement, everything we discussed so far still holds with just two differences:
\begin{itemize}
\item[\textbullet] if we extract twice or more times the same element, its extended probability flips sign the first time, and then stays constant;
\item[\textbullet] the convergence may be slower, because we may need more extractions from the urn to learn the true composition (because we can extract twice or more times the same element).
\end{itemize}

Let us give a simple example that illustrates one of the possible situations in which the analysis depicted so far can be put to use.

\begin{example}\label{ex_species}
We consider a species sampling problem in the  field of ecology. Suppose we want to know the number of bird species that inhabit a certain region throughout the year. What we can do is to start our analysis by letting $\Omega=\mathbb{N}$ and $\Omega_0^+=\{1,\ldots,n\}$, where $n$ is the number of species that  inhabit a region similar to the one of interest throughout the year. Then, we specify a set $\mathcal{P}$ of probability measures on $\mathbb{N}$, e.g. a collection $\{\text{Geom}(p)\}_{p\in [0,1]}$ of geometric distributions having parameter $p\in [0,1]$. We specify a set  of probabilities because we are not able to express our initial opinion via a unique probability measure (we face ambiguity). We specify the probability measures on the whole number field $\mathbb{N}$ because we do not know exactly the composition of our state space. Then, after eliciting the set $\mathcal{P}^{ex}_0$ of extended probability measures induced by (the elements of) $\mathcal{P}$, we begin the ex ante analysis described in the present work. After collecting observations for an entire year, we end up discovering that the state space associated with our experiment is $\Omega^\prime=\{1,\ldots,m\} \subsetneq \mathbb{N}$, where $m \geq n$ is the number of species we counted during the year. We recover also the set $\tilde{\mathcal{P}}$ of regular probability measures induced by $\mathcal{P}^{ex}_\infty$. Now, the ``real'' statistical analysis can take place. Indeed, we know that $m$ is the {maximum} number of species that live in the region during the year, but it does not take into account migrations to or from the region itself. Hence,  the number of bird species that inhabit the region throughout the year may  well be smaller than $m$. So, every $\tilde{P} \in \tilde{\mathcal{P}}$ will be such that $\tilde{P}(\{\omega_k\})\in [0,1]$, and $\sum_{k=1}^m \tilde{P}(\{\omega_k\})=1$, where $\omega_k=k$, for all $k \in \{1,\ldots,m\}$. Gathering data $y_1,\ldots,y_\ell$ and updating these $\tilde{P}$'s via Bayesian conditioning, we obtain the set $\{\tilde{P}(\cdot \mid y_1,\ldots,y_\ell)\}_{\tilde{P} \in \tilde{\mathcal{P}}}$ of posterior (regular) probability measures. This gives a robust analysis: for all $\omega_k \in \Omega^\prime$, the posterior probability of $\omega_k$ being the correct number of species belongs to the interval 
$$\left[ \underline{\tilde{P}}(\{\omega_k\} \mid y_1,\ldots,y_\ell),\overline{\tilde{P}}(\{\omega_k\} \mid y_1,\ldots,y_\ell) \right],$$
where the lower bound is a lower (regular) probability, and the upper bound is an upper (regular) probability. The narrower the interval, the less imprecise our beliefs resulting from the analysis.
\demo
\end{example}

A criticism that can be made of our ex ante analysis is the following: why should the scholar be concerned with the exact composition of the state space? It should be enough to specify it as richly as possible, and then proceed to a regular statistical analysis. There are two responses to such a critique. First, as pointed out in Remark \ref{rem_on_omega}, in doing so the scholar would not be respecting the condition listed in \cite{tsitsiklis} that the state space must have the right granularity depending on the statistical experiment they are interested in. Second, in the case that the state space associated with our statistical experiment is $\Omega^\prime \subsetneq \Omega$, working with probability measures supported on the whole space $\Omega$ of apparently possible states may be computationally costly. Our ex ante analysis allows the scholar to focus on the ``minimal'' state space -- the one containing only the necessary states -- that is more meaningful to the analysis.

\section{Upper and lower extended probabilities}\label{ulep}
Fix any $t \in \mathbb{N}_0$, and consider the ``boundary elements'' of the set $\mathcal{P}^{ex}_t$,
\begin{align}\label{elp}
\underline{P}^{ex}_t(A)=\inf\limits_{P^{ex} \in \mathcal{P}^{ex}_t} P^{ex}(A)
\end{align}
and
\begin{align}\label{eq7}
\overline{P}^{ex}_t(A)=\sup\limits_{P^{ex} \in \mathcal{P}^{ex}_t}{P}^{ex}(\Omega)-\underline{P}^{ex}_t(A^c)=\sup\limits_{P^{ex} \in \mathcal{P}^{ex}_t} P^{ex}(A),
\end{align} 
for all $A \in \mathcal{F}$. They are not extended probabilities; we call $\underline{P}^{ex}_t(A)$ a {lower extended probability measure}, and $\overline{P}^{ex}_t(A)$ an {upper extended probability measure}. Notice that upper extended probability measures differ from upper regular probability measures. These latter are defined as $1$ minus the lower regular probability of the complement of the event we are interested in. In \eqref{eq7} we give a similar conjugate type of definition, but we cannot write $1$, because we do not require that the lower extended probability of $\Omega$ is $1$.

Given a generic lower extended probability $\underline{P}^{ex}$ and a generic event $A \in\mathcal{F}$, we interpret $\underline{P}^{ex}(A)$ as follows. If $\underline{P}^{ex}(A)=p>0$, then $p$ represents the supremum price that we are willing to pay to enter a bet on $A$ that gives us \$$1$ if $A$ takes place and \$$0$ otherwise. If $\underline{P}^{ex}(A)=0$, then in the most conservative of our mental states we deem $A$ impossible to take place. Finally, if $\underline{P}^{ex}(A)=q<0$, then $|q|$ represents the infimum selling price at which our doppelgänger takes bets on $A$ that pay \$$1$ if $A$ takes place and \$$0$ otherwise. 
As we can see, this interpretation captures the ideas of worst case scenario and of prudent behavior. It can be seen as a betting scheme analogous to the one in \cite[section 2.3.1]{walley}, but with monetary instead of utiles outcomes, and extended to probabilities that can take on negative values as well.

Both lower and upper extended probabilities are {extended Choquet capacities}. A generic extended Choquet capacity is defined as a set function $\nu^{ex}:\mathcal{F} \rightarrow \mathbb{R}$ such that
\begin{itemize}
\item[(EC1)] $\nu^{ex}(\emptyset)=0$,
\item[(EC2)] $\nu^{ex}(A) \in [-1,1]$, for all $A \in \mathcal{F}$, 
\item[(EC3)] for any $A, B \in \mathcal{F}$ such that $A \subset B$, if $\nu^{ex}(A) \geq 0$, $\nu^{ex}(B) \geq 0$, and $\nu^{ex}(B \cap A^c) \geq 0$, then $\nu^{ex}(A) \leq \nu^{ex}(B)$. If instead $\nu^{ex}(A) \leq 0$, $\nu^{ex}(B) \leq 0$, and $\nu^{ex}(B \cap A^c) \leq 0$, then $\nu^{ex}(A) \geq \nu^{ex}(B)$.
\end{itemize}
As we can see, we do not require countable additivity (ii*) to hold for extended capacities. An extended probability measure is an additive extended capacity.

It is easy to see that upper and lower extended probability measures satisfy (EC1)-(EC3); in addition, lower extended probabilities are superadditive, while upper extended probabilities are subadditive. That is, for all $A,B \in \mathcal{F}$,
\begin{equation}\label{eq55}
\underline{P}^{ex}_t(A \sqcup B) \geq \underline{P}^{ex}_t(A)+\underline{P}^{ex}_t(B)
\end{equation}
and
\begin{equation}\label{eq56}
\overline{P}^{ex}_t(A \sqcup B) \leq \overline{P}^{ex}_t(A)+\overline{P}^{ex}_t(B).
\end{equation}
These inequalities come immediately from the properties of the infimum and supremum operators. 

\begin{remark}\label{lower_coherence}
Notice that the behavioral interpreatation that we gave to lower extended probabilities entails that they can be specified even without eliciting a set of extended probability measures first. To this extent, our behavioral interpretation can be called minimal, similarly to \cite[section 2.3.1]{walley}. An immediate question the reader may ask is: ``If we were to specify a lower extended probability without resorting to a set of extended probabilities, are we sure it is subadditive?''. The answer is yes, under a mild assumption: Example \ref{ex_coh} -- based on the example in \cite[section 1.6.4]{walley} -- shows that if lower extended probabilities avoid sure loss, then they are superadditive, and Theorem \ref{lepc} shows that if $\underline{P}^{ex}$ can be obtained as the infimum of a set of extended probabilities, then it avoids sure loss.

We first define \textit{sure loss} for lower extended probabilities. It is the immediate lower counterpart of Definition \ref{coherence}.
\begin{definition}\label{lower_coh_def}
Lower extended probability $\underline{P}^{ex}$ avoids sure loss if no Dutch books can be made against the punter, that is, if we cannot find a finite collection $\{B_j\}_{j=1}^n \subset \mathscr{B}^\prime$ along with numbers $\{s_j\}_{j=1}^n \subset \mathbb{R}_+$ such that
\begin{align}\label{eq27_2}
\sup_{\omega\in\Omega}\left\lbrace{\sum_{j=1}^n s_j \left[ \mathbbm{1}_{B_j}(\omega)-\underline{P}^{ex}(B_j) \right]}\right\rbrace<0,
\end{align}
where $\mathscr{B}^\prime$ is the sigma-algebra generated by the collection $\{B \subset \Omega: \underline{P}^{ex}(B)\geq 0 \text{ and } B\cap C = \emptyset \text{, } \forall C \in\mathcal{F} \text{, } \underline{P}^{ex}(C)<0\}$.

\end{definition}
We have the following interesting result, that is a version of the lower envelope theorem in  \cite[Corollary 2.8.6]{walley}.\footnote{Notice that the lower envelope theorem in  \cite[Corollary 2.8.6]{walley} comprises an envelope of probabilities that are merely finitely additive, and it states a necessary and sufficient condition. We plan to prove the opposite direction of Theorem \ref{lepc} in future work.}
\begin{theorem}\label{lepc}
Consider a generic lower extended probability $\underline{P}^{ex}$ defined on a measurable space $(\Omega,\mathcal{F})$. If there exists a nonempty set $\mathcal{P}^{ex}$ of extended probabilities on $(\Omega,\mathcal{F})$ such that $\underline{P}^{ex}(A)=\inf_{P^{ex} \in \mathcal{P}^{ex}}P^{ex}(A)$, for all $A \in \mathcal{F}$, then $\underline{P}^{ex}$ avoids sure loss.
\end{theorem}
The following example shows that if $\underline{P}^{ex}$ avoids sure loss, then it is superadditive.
\begin{example}\label{ex_coh}
Pick two mutually exclusive events $A$ and $B$. By our behavioral interpretation, the highest amount we are willing to pay to get \$$1$ if $A$ occurs (or the highest amount we would deem reasonable to lose if we were given the possibility to enter the bet) is $\underline{P}^{ex}(A)$. The same holds for $\underline{P}^{ex}(B)$. Because we want to avoid sure loss, the net outcome is equivalent to paying $\underline{P}^{ex}(A) + \underline{P}^{ex}(B)$ to get \$$1$ if $A \sqcup B$ occurs. Now, $\underline{P}^{ex}(A \sqcup B)$ is the highest price we are willing to pay to obtain \$$1$ if $A \sqcup B$ occurs. So we recover the superadditivity constraint
$$\underline{P}^{ex}(A \sqcup B) \geq \underline{P}^{ex}(A) + \underline{P}^{ex}(B).$$
Similarly, an upper lower probability $\overline{P}^{ex}$ should satisfy the subadditivity constraint $\overline{P}^{ex}(A \sqcup B) \leq \overline{P}^{ex}(A) + \overline{P}^{ex}(B)$.
\demo
\end{example}
\end{remark}

An important concept worth introducing is the {core} of a lower extended probability measure.
\begin{definition}\label{core}
Given a generic lower extended probability $\underline{P}^{ex}$, we call {core} of $\underline{P}^{ex}$ the set
$$\text{core}(\underline{P}^{ex}):=\left\lbrace{P^{ex} \in \Delta^{ex}(\Omega,\mathcal{F}) : P^{ex}(A) \geq \underline{P}^{ex}(A) \text{, } \forall A \in \mathcal{F} \text{ and } P^{ex}(\Omega) = \underline{P}^{ex}(\Omega) }\right\rbrace.$$
\end{definition}
This definition tells us that the core of $\underline{P}^{ex}$ is the set of all suitably normalized extended probabilities that setwise dominate $\underline{P}^{ex}$. Notice that in general it may be empty, and that
\begin{align*}
\text{core}(\underline{P}^{ex}) &= \left\lbrace{ P^{ex} \in \Delta^{ex}(\Omega,\mathcal{F}) : \underline{P}^{ex} \leq {P}^{ex} \leq \overline{P}^{ex} }\right\rbrace\\
&= \left\lbrace{ P^{ex} \in \Delta^{ex}(\Omega,\mathcal{F}) : P^{ex}(A) \leq \overline{P}^{ex}(A) \text{, } \forall A \in \mathcal{F} \text{ and } P^{ex}(\Omega) = \underline{P}^{ex}(\Omega) }\right\rbrace,
\end{align*}
so the core can be seen as the set of extended probabilities ``sandwiched'' between lower extended probability $\underline{P}^{ex}$ and upper extended probability $\overline{P}^{ex}$, as well as the set of extended probabilities setwise dominated by $\overline{P}^{ex}$. The following is a corollary to Theorem \ref{lepc}.
\begin{corollary}\label{cor_coh}
If $\text{core}(\underline{P}^{ex})\neq \emptyset$, then $\underline{P}^{ex}$ avoids sure loss.
\end{corollary}
A crucial property of the core is the following.
\begin{proposition}\label{core_cpct}
Given a generic lower extended probability $\underline{P}^{ex}$, its core is convex and weak$^\star$-compact.
\end{proposition}
Being compact and convex, the core of $\underline{P}^{ex}$ is completely characterized by $\underline{P}^{ex}$. This means that it is enough to know $\underline{P}^{ex}$ to be able to retrieve every element in its core. So in our analysis we can focus on updating $\underline{P}^{ex}_t$ to $\underline{P}^{ex}_{t+1}$, and then require that $\mathcal{P}^{ex}_{t+1}=\text{core}(\underline{P}^{ex}_{t+1})$, instead of updating $\mathcal{P}^{ex}_{t}$ to $\mathcal{P}^{ex}_{t+1}$ elementwise. This justifies our focus in the remainder of this section on studying how to update lower extended probabilities. The procedure to update upper extended probability measures is going to be similar (their relation is described by equation \eqref{eq7}).

Recall that the conditions to perform the update in the additive case are given by (\ref{eq4_1}) and (\ref{eq41}). 
\begin{proposition}\label{prop12}
The sublinear counterpart of (\ref{eq4_1}) is the following,
\begin{align}\label{eq8}
\underline{P}_{t+1}^{ex}(E^+_{t+1,\omega})=\left| \overline{P}_t^{ex}(E^-_{t,\omega}) \right| \quad \text{and} \quad \overline{P}_{t+1}^{ex}(E^+_{t+1,\omega})=\left| \underline{P}_t^{ex}(E^-_{t,\omega}) \right|.
\end{align}
\end{proposition}
It holds when we learn a new element $\omega$ of our sample space, that is, when the new observation $\omega$ belongs to the latent space at time $t$, but ``moves'' to the actual space at time $t+1$. In formulas, $\omega \in \Omega^-_{t}$, but $\omega \in \Omega^+_{t+1}$. The lower extended probabilities assigned to the other elements of $\mathcal{E}$ are held constant.

Notice that we are working with the finest possible partition of $\Omega$, so, as before, the intersection $A \cap E$ between any $A \in \mathcal{F}$ and any $E=\{\omega\} \in \mathcal{E}$ is either the empty set, $A \cap E= \emptyset$, if $\omega \not\in A$, or the element $\omega$ itself, $A \cap E=\{\omega\}=E$, if $\omega \in A$. Hence, for any $t \in \mathbb{N}_0$, for any $E=\{\omega\} \in \mathcal{E}$, and for any $A \in \mathcal{F}$, we have that
\begin{align}\label{eq43}
\underline{P}^{ex}_t(A \cap E)=\begin{cases}
\underline{P}^{ex}_t(E) & \text{if } \omega \in A\\
\underline{P}^{ex}_t(\emptyset)=0 & \text{if } \omega \not\in A
\end{cases}.
\end{align}

Given any $A,B \in \mathcal{F}$, if $\underline{P}_t^{ex}(B) \neq 0$, we define
\begin{align}\label{eq20}
\underline{P}_t^{ex}(A \mid B):=\frac{\inf_{P^{ex} \in \mathcal{P}^{ex}_t} P^{ex}(A \cap B)}{\inf_{P^{ex} \in \mathcal{P}^{ex}_t} P^{ex}(B)}=\frac{\underline{P}_t^{ex}(A \cap B)}{\underline{P}_t^{ex}(B)},
\end{align}
for all $t \in \mathbb{N}_0$. We call it the \textit{extended Geometric rule}. We notice immediately that, combining \eqref{eq43} and \eqref{eq20}, we get
\begin{align}\label{eq44}
\underline{P}^{ex}_t(A \mid E)=\frac{\underline{P}_t^{ex}(A \cap E)}{\underline{P}_t^{ex}(E)}=\begin{cases}
\frac{\underline{P}_t^{ex}(E)}{\underline{P}_t^{ex}(E)}=1 & \text{if } \omega \in A\\
\frac{\underline{P}_t^{ex}(\emptyset)}{\underline{P}_t^{ex}(E)}=0 & \text{if } \omega \not\in A
\end{cases},
\end{align}
for all $E=\{\omega\} \in \mathcal{E}$ such that $\underline{P}_t^{ex}(E) \neq 0$ and all $A \in \mathcal{F}$. So, we have that $\underline{P}^{ex}_t(A \mid E)={P}^{ex}_t(A \mid E)$, for all $A$, all $E$, and all $P^{ex}_t \in \mathcal{P}^{ex}_t$. This is true only because we are working with the finest possible partition of $\Omega$, and because we endorse the extended Geometric rule.

The sublinear counterpart of (\ref{eq41}) is the following
\begin{align}\label{eq9}
\underline{P}^{ex}_{t+1}(E^+_{t+1,\omega}) = \underline{P}^{ex}_t(E^+_{t,\omega}) \geq 0,
\end{align}
for all $t \in \mathbb{N}_0$. This comes from the fact that we draw an element already belonging to the actual space, so we do not need to update its lower extended probability (similarly to what is described in equation \eqref{eq41}). The lower extended probabilities assigned to the other elements of $\mathcal{E}$ are held constant.

At this point, a natural question one may ask is how to compute $\underline{P}^{ex}_{t}(A)$, for any $t$, for any $A \in \mathcal{F}$. It would be tempting to write that
\begin{align*}
\begin{split}
\underline{P}^{ex}_{t}(A)=&\sum\limits_{E^+_{t,j} \in \mathcal{E}^+_t : \underline{P}^{ex}_{t}(E^+_{t,j})\neq 0} \underline{P}^{ex}_{t}(A \mid E^+_{t,j}) \underline{P}^{ex}_{t}(E^+_{t,j})\\ &+ \sum\limits_{E^-_{t,j} \in \mathcal{E}^-_t : \underline{P}^{ex}_{t}(E^-_{t,j})\neq 0} \underline{P}^{ex}_{t}(A \mid E^-_{t,j}) \underline{P}^{ex}_{t}(E^-_{t,j}).
\end{split}
\end{align*}
This would mimic exactly \eqref{eq5_1}, with lower extended probabilities in place of additive extended probabilities. Alas, that would not be true, since lower extended probabilities are not additive. Instead, we have the following.

Fix any $t \in \mathbb{N}_0$. Call $\mathfrak{E}_t \equiv \{E_{t,A}\}$ the collection of elements of $\mathcal{E}$ such that $A \cap E_{t,A} \neq \emptyset$. We index $\mathfrak{E}_t$ to time $t$ to highlight the fact that although its elements stay the same, some of them may ``move'' from the latent to the actual space as time goes by and we collect more obsevations. Since we are working with the finest possible partition of $\Omega$, we can write
\begin{equation}\label{eq50}
A=\bigsqcup_{E_{t,A} \in \mathfrak{E}_t} E_{t,A}.
\end{equation}
In the most general case, some of these $E_{t,A}$'s belong to the actual space, and some to the latent space. Let us denote the former by $E_{t,A}^+$'s and the latter by $E_{t,A}^-$'s. Formally, we have  that $E_{t,A}^+ \cap \Omega^+_t \neq \emptyset$ and $E_{t,A}^+ \cap \Omega^-_t = \emptyset$, and vice versa for the $E_{t,A}^-$'s. Hence, \eqref{eq50} can be rewritten as
\begin{equation}\label{eq51}
A=\bigsqcup_{E_{t,A}^+ \in \mathfrak{E}_t} E_{t,A}^+\sqcup \bigsqcup_{E_{t,A}^- \in \mathfrak{E}_t} E_{t,A}^-.
\end{equation}
Now, from \eqref{eq8} and \eqref{eq9}, we know how to update the lower extended probabilities of all the  elements of the patition $\mathcal{E}$ of $\Omega$, which implies that we know the value of $\underline{P}^{ex}_{t+1}(E_{t+1,A})$, for all $E_{t+1,A} \in \mathfrak{E}_{t+1}$. Then, consider now the summation 
$$\sum_{E_{t+1,A} \in \mathfrak{E}_{t+1}} \underline{P}^{ex}_{t+1}(E_{t+1,A}).$$ 
From equation \eqref{eq55}, we have that 
\begin{align*}
 \underline{P}^{ex}_{t+1}(A) &\geq \sum_{E_{t+1,A} \in \mathfrak{E}_{t+1}} \underline{P}^{ex}_{t+1}(E_{t+1,A})\\ &= \sum_{E_{t+1,A}^+ \in \mathfrak{E}_{t+1}} \underline{P}^{ex}_{t+1}(E^+_{t+1,A}) + \sum_{E_{t+1,A}^- \in \mathfrak{E}_{t+1}} \underline{P}^{ex}_{t+1}(E^-_{t+1,A}).   
\end{align*}
From equation \eqref{eq56}, we can give an upper bound for the upper extended probability of $A$,
\begin{align*}
    \overline{P}^{ex}_{t+1}(A) &\leq \sum_{E_{t+1,A} \in \mathfrak{E}_{t+1}} \overline{P}^{ex}_{t+1}(E_{t+1,A})\\ &= \sum_{E_{t+1,A}^+ \in \mathfrak{E}_{t+1}} \overline{P}^{ex}_{t+1}(E^+_{t+1,A}) + \sum_{E_{t+1,A}^- \in \mathfrak{E}_{t+1}} \overline{P}^{ex}_{t+1}(E^-_{t+1,A}).
\end{align*}

\section{Conclusion}\label{conclusion}
In this work  we give a definition of extended probability measures that does not depend on the field a scholar works in. We give some of their more interesting properties, and a behavioral interpretation to positive and negative values of extended probabilities. We then apply extended probabilities to statistical inference. Given the probability space $(\Omega,\mathcal{F},P)$ associated with the experiment we want to conduct, we use extended probabilities to express uncertainty about both the composition of the state space $\Omega$, and which probability measure $P$ to select. We develop an ex-ante analysis; our method describes how the researcher progressively discovers the true composition of the state space, so that, at the end of the process, a regular statistical analysis (that requires the knowledge of $\Omega$) can take place. We introduce the concept of extended Choquet capacities, and in particular of upper and lower extended probabilities that represent the ``borders'' of sets of extended probabilities, and we give bounds for the lower extended probability of any element $A \in \mathcal{F}$. We also apply our model to the fields of opinion dynamics and species sampling models. This paper is important because it gives a foundational definition of extended probability measures and makes these latter relevant to the field of statistics. 

In the future, we will deal with the possibility that in our ex-ante analysis we are not able to discover the whole composition of the state space associated with our experiment. In that case, the statistical analysis itself will have to be carried out using extended probabilities. We also plan to relax the assumption we made about working with a finite or countable state space. We will also provide direct behavioral motivation for the properties of extended probabilities so as to be able to relate our approach to those of de Finetti \cite{definetti1,definetti2} and Walley \cite{walley}. Furthermore, since it is one of the main benefits of a betting approach, we will study partially specified extended probabilities (that is, extended probabilities defined not on a sigma-algebra, but rather on a generic collection of events) and their coherent extension. Finally, we would like to deepen the study of lower extended probabilities.

\section*{Acknowledgments}\label{ackn}
This work is supported in part by the following grants: NSF CCF-193496, NSF DEB-1840223, NIH R01 DK116187-01, HFSP RGP0051/2017, NIH R21 AG055777-01A, NSF DMS 17-13012, NSF ABI 16-61386, ARO MURI W911NF2010080. The authors would like to thank Mark Burgin, Shounak Chattopadhyay, Roberto Corrao, Vittorio Orlandi, Glenn Shafer, Peter Walley, Marco Zaffalon and Alessandro Zito for helpful comments.


\appendix

\section{Proofs}\label{proofs}
\begin{proof}[Proof of Proposition \ref{prop1}]
Suppose for the sake of contradiction that for some $A=B$, 
\begin{align}\label{eq23}
P^{ex}(A) \neq P^{ex}(B).
\end{align}
Consider then $A \setminus B =  B \setminus B = \emptyset$; we have that
\begin{align}\label{eq11}
P^{ex}(A \setminus B) \neq P^{ex}(B \setminus B)=P^{ex}(\emptyset)=0.
\end{align}
This is true because $A=\{A \setminus B\} \sqcup \{A \cap B\}$ and $B=\{B \setminus B\} \sqcup \{B \cap B\}$; by (ii*), this implies that $P^{ex}(A \setminus B)=P^{ex}(A)-P^{ex}(A \cap B)$ and $P^{ex}(B \setminus B)=P^{ex}(B)-P^{ex}(B \cap B)$. Then, we have that $P^{ex}(A \setminus B) = P^{ex}(B \setminus B)$ if and only if
$$P^{ex}(A)=P^{ex}(B)-P^{ex}(B \cap B)+P^{ex}(A \cap B);$$
but $B \cap B=B$, and $A \cap B=B$ because we assumed $A=B$, so $P^{ex}(A \setminus B) = P^{ex}(B \setminus B)$ if and only if $P^{ex}(A)=P^{ex}(B)$. But by \eqref{eq23} $P^{ex}(A)\neq P^{ex}(B)$, so we have that the inequality in \eqref{eq11} holds.

By assumption we know that $A=B$, so $A \setminus B = \emptyset$, and so $P^{ex}(A\setminus B)=P^{ex}(\emptyset)$. Then, by \eqref{eq11}, we have that $P^{ex}(\emptyset)=P^{ex}(A\setminus B)\neq P^{ex}(B\setminus B)=P^{ex}(\emptyset)$, a contradiction.
\end{proof}

\begin{proof}[Proof of Proposition \ref{prop2}]
Let $A \subset B$; then, we can write $B$ as $B=A \sqcup (B \cap A^c)$. By (ii*) this implies that $P^{ex}(B)=P^{ex}(A)+P^{ex}(B \cap A^c)$. Then, if $P^{ex}(B)$, $P^{ex}(A)$, and $P^{ex}(B \cap A^c)$ are all nonnegative, then $P^{ex}(B)=P^{ex}(A)+P^{ex}(B \cap A^c) \geq P^{ex}(A)$. If instead they are all nonpositive, then $P^{ex}(B)=P^{ex}(A)+P^{ex}(B \cap A^c) \leq P^{ex}(A)$.
\end{proof}

\begin{proof}[Proof of Proposition \ref{prop3ext}]
Pick $A,B \in \mathcal{F}$; if they are disjoint, the equation follows immediately from (ii*). If they are not, consider $A \cup B=\{A \setminus B\} \sqcup \{A \cap B\} \sqcup \{B \setminus A\}$. Then, again by (ii*), we have that 
\begin{align}\label{eq12}
P^{ex}(A \cup B)=P^{ex}(A \setminus B)+P^{ex}(A \cap B)+P^{ex}(B \setminus A)
\end{align}
Notice then that $P^{ex}(A \setminus B)=P^{ex}(A)-P^{ex}(A\cap B)$; indeed
$$P^{ex}(A)=P^{ex}(\{A\cap B\} \sqcup \{A \setminus B\})=P^{ex}(A\cap B)+P^{ex}(A \setminus B).$$
Similarly, $P^{ex}(B \setminus A)=P^{ex}(B)-P^{ex}(B\cap A)$. So, \eqref{eq12} becomes
\begin{align*}
P^{ex}(A \cup B) &= P^{ex}(A)-P^{ex}(A\cap B) + P^{ex}(A\cap B) + P^{ex}(B)-P^{ex}(B\cap A) \\&= P^{ex}(A)+P^{ex}(B)-P^{ex}(A\cap B).
\end{align*}
\end{proof}

\begin{proof}[Proof of Theorem \ref{alw_coh}]
Notice 
that $P^{ex}$ restricted to $\mathscr{B}$ is a finite signed measure. By the Hahn-Jordan decomposition theorem \cite{fischer}, there exists a unique decomposition of $P^{ex}$ into a difference $P^{ex}=P^{ex}_+ - P^{ex}_-$ of two finite positive measures. Pick any $\{B_j\}_{j=1}^n \subset \mathscr{B}$, $\{s_j\}_{j=1}^n \subset \mathbb{R}$. We have that, for any payoff
\begin{align*}
\mathfrak{f}(\omega)&=\sum_{j=1}^n s_j \left[ \mathbbm{1}_{B_j}(\omega) - P^{ex}(B_j) \right]\\
&=\sum_{j=1}^n s_j \left[ \mathbbm{1}_{B_j}(\omega) - P^{ex}_+(B_j)+P^{ex}_-(B_j) \right],
\end{align*}
the following holds
\begin{align*}
\int_\Omega \mathfrak{f} \mathrm{d}P^{ex}&=\sum_{j=1}^n s_j \left[ \int_\Omega \mathbbm{1}_{B_j}\mathrm{d}P^{ex} - P^{ex}_+(B_j)+P^{ex}_-(B_j) \right]\\
&=\sum_{j=1}^n s_j \left[ \int_\Omega \mathbbm{1}_{B_j}\mathrm{d}P^{ex}_+ - \int_\Omega \mathbbm{1}_{B_j}\mathrm{d}P^{ex}_- - P^{ex}_+(B_j)+P^{ex}_-(B_j) \right]\\
&=\sum_{j=1}^n s_j \left[ \int_\Omega \mathbbm{1}_{B_j}\mathrm{d}P^{ex}_+ - P^{ex}_+(B_j) \right] - \sum_{j=1}^n s_j \left[ \int_\Omega \mathbbm{1}_{B_j}\mathrm{d}P^{ex}_- - P^{ex}_-(B_j) \right] = 0.
\end{align*}
So $\mathfrak{f}$ cannot have a negative supremum.
\end{proof}

\begin{proof}[Proof of Proposition \ref{prop4}]
We have that
\begin{align}\label{eq14_bis}
\begin{split}
P^{ex}_t(A) &= P^{ex}_t(A \cap \Omega_t^+) + P^{ex}_t(A \cap \Omega_t^-)\\
&=\sum\limits_{E_{t,j}^+ \in \mathcal{E}_t^+} P^{ex}_t ( A \cap E_{t,j}^+ ) + \sum\limits_{E_{t,j}^- \in \mathcal{E}_t^-} P^{ex}_t ( A \cap E_{t,j}^- )\\
&=\sum\limits_{E_{t,j}^+ \in \mathcal{E}_t^+ : P^{ex}_t(E_{t,j}^+) \neq 0} P^{ex}_t ( A \mid E_{t,j}^+ )P^{ex}_t(E_{t,j}^+)\\ &+ \sum\limits_{E_{t,j}^- \in \mathcal{E}_t^- : P^{ex}_t(E_{t,j}^-) \neq 0} P^{ex}_t ( A \mid E_{t,j}^- )P^{ex}_t(E_{t,j}^-),
\end{split}
\end{align}
where the first equality comes from $A=(A \cap \Omega_t^+)\sqcup(A \cap \Omega_t^-)$ and (ii*), the third equality comes from \eqref{eq13}, and the second equality comes from $\Omega^+_t = \sqcup_{E_{t,j}^+ \in \mathcal{E}_t^+} E_{t,j}^+$, so $$A \cap \Omega^+_t = A \cap (\sqcup_{E_{t,j}^+ \in \mathcal{E}_t^+} E_{t,j}^+ )=\sqcup_{E_{t,j}^+ \in \mathcal{E}_t^+} (A \cap E_{t,j}^+),$$
where the last equality comes from De Morgan laws.
\end{proof}

\begin{proof}[Proof of Proposition \ref{prop7}]
Pick any $A,B \in\mathcal{F}$. Since $\Omega_t^+ \cap \Omega^-_t=\emptyset$, we have that
\begin{align*}
A\cup B &= ((A\cup B)\cap \Omega^+_t)\sqcup ((A\cup B)\cap \Omega^-_t)\\
&= ((A\cap \Omega^+_t)\cup(B\cap \Omega^+_t)) \sqcup ((A\cap \Omega^-_t)\cup(B\cap \Omega^-_t))\\
&=(A^+_t \cup B^+_t)\sqcup (A^-_t \cup B^-_t).
\end{align*}
A similar argument shows that $A \cap B=(A^+_t \cap B^+_t)\cup (A^-_t \cap B^-_t)$, for all $t$.

For the second part, notice that $A \setminus B=\{\omega \in \Omega : \omega \in A \text{, } \omega \not\in B\}$. Then, $A \setminus B \cap  \Omega^+_t \equiv A^+_t \setminus B^+_t=\{\omega \in \Omega_t^+ : \omega \in A\text{, } \omega \not\in B\}$; similarly, $A \setminus B \cap  \Omega^-_t \equiv A^-_t \setminus B^-_t=\{\omega \in \Omega_t^- : \omega \in A, \omega \not\in B\}$. But $\Omega_t^+ \sqcup \Omega_t^-=\Omega$, so the claim follows.
\end{proof}

\begin{proof}[Proof of Proposition \ref{prop8}]
Let us focus on $\mathcal{F}_t^+$; it is closed with respect to countable intersections because it is a sigma-algebra. Then, pick $A,B \in \mathcal{F}^+_t$ such that $A \neq B$. If $A \cap B = \emptyset$, then $A\setminus B=A \in \mathcal{F}^+_t$; if $A \cap B \neq \emptyset$, then $A\setminus B=A \cap B^c$. Then, $B \in \mathcal{F}^+_t$ implies $B^c \in \mathcal{F}^+_t$ because $\mathcal{F}_t^+$ is a sigma-algebra; also, $A \cap B^c \in \mathcal{F}_t^+$ because $\mathcal{F}_t^+$ is closed with respect to countable intersections. So $A\setminus B \in \mathcal{F}_t^+$. Finally, the unit element is $\Omega^+_t$: for any $A \in \mathcal{F}^+_t$, $A \cap \Omega^+_t=A$, because $A \in\mathcal{F}^+_t$. Hence $\mathcal{F}^+_t$ is a set algebra. We show in a similar fashion that $\mathcal{F}^-_t$ is a set algebra as well.
\end{proof}

\begin{proof}[Proof of Proposition \ref{prop9}]
To ease notation, let $C \equiv \cup_{j \in \mathbb{N}_0} A_j$. $C$ belongs to $\mathcal{F}$, because $\mathcal{F}$ is a sigma-algebra. Then, let $C \in\mathcal{F}_t^+$. This implies that $P^{ex}_t(C) \geq 0$, but also that $P^{ex}(A) \geq 0$ and that $P^{ex}(C \cap A^c) \geq 0$, since $C=A \sqcup (C \cap A^c)$. Then, 
$$P^{ex}_t(A) \leq P^{ex}_t(C) \leq \sum_{j \in \mathbb{N}_0}P^{ex}_t(A_j),$$
where the first inequality comes from Proposition \ref{prop2}, and the second one from Proposition \ref{prop3ext}.

If $C \in\mathcal{F}_t^-$, then $P^{ex}_t(C) \leq 0$, and also that $P^{ex}(A) \leq 0$ and that $P^{ex}(C \cap A^c) \leq 0$. Then, 
$$P^{ex}_t(A) \geq P^{ex}_t(C) \geq \sum_{j \in \mathbb{N}_0}P^{ex}_t(A_j),$$
where again the first inequality comes from Proposition \ref{prop2}, and the second one from Proposition \ref{prop3ext}.

If $C\equiv \cup_{j \in \mathbb{N}_0} A_j$ is such that $C\cap \Omega_t^+\neq \emptyset \neq C\cap \Omega_t^-$, then we cannot say anything general about the relation between $P^{ex}_t(A)$ and $\sum_{j \in \mathbb{N}_0}P^{ex}_t(A_j)$.
\end{proof}

\begin{proof}[Proof of Proposition \ref{prop11}]
Notice that by \eqref{guidance_p} and\eqref{guidance_p2}, we have that, for all $t$,
\begin{equation*}
P^{ex}_t(A\cap\Omega^+_t)=P(A\cap\Omega^+_t)
\end{equation*}
and
\begin{equation*}
P^{ex}_t(A\cap\Omega^-_t)=-P(A\cap\Omega^-_t),
\end{equation*}
for all $A \in \mathcal{F}$, where $P \in \Delta(\Omega,\mathcal{F})$. Now, to save some notation, let us denote by $\Omega^+ \equiv \cup_{t \in \mathbb{N}_0} \Omega_t^+$ and by $\Omega^- \equiv \cap_{t \in \mathbb{N}_0} \Omega_t^-$. Then, we have the following.
\begin{align*}
d_{ETV}(P^{ex}_t,P^{ex}_\infty) &= \sup_{A \in \mathcal{F}} |P^{ex}_t(A)-P^{ex}_\infty(A)| \equiv |P^{ex}_t(\mathbf{A})-P^{ex}_\infty(\mathbf{A})|.
\end{align*}
Let us denote by $\mathbf{A}^+ \equiv \mathbf{A} \cap \Omega^+$, and by $\mathbf{A}^- \equiv \mathbf{A} \cap \Omega^-$, so $\mathbf{A}=\mathbf{A}^+ \sqcup \mathbf{A}^-$. Notice that even though in the limit we fully discover $\mathbf{A}^+$, there may be some $t \in \mathbb{N}_0$ such that $\mathbf{A}^+ \cap \Omega^+_t \neq \emptyset \neq \mathbf{A}^+ \cap \Omega^-_t$. So we have 
\begin{align}
|P^{ex}_t(\mathbf{A})-P^{ex}_\infty(\mathbf{A})| &= |P^{ex}_t(\mathbf{A}^+)+P^{ex}_t(\mathbf{A}^-)-P^{ex}_\infty(\mathbf{A}^+)-P^{ex}_\infty(\mathbf{A}^-)| \label{from_eq_2}\\
&=|P^{ex}_t(\mathbf{A}^+)-P^{ex}_\infty(\mathbf{A}^+)| \label{alw_latent}\\
&=|P(\mathbf{A}^+ \cap \Omega^+_t)-P(\mathbf{A}^+ \cap \Omega^-_t)-P(\mathbf{A}^+ \cap \Omega^+_t)-P(\mathbf{A}^+ \cap \Omega^-_t)| \label{by_constr}\\
&=2 P(\mathbf{A}^+ \cap \Omega^-_t) \nonumber.
\end{align}
Equation \eqref{from_eq_2} comes from $\mathbf{A}=\mathbf{A}^+ \sqcup \mathbf{A}^-$ and (ii*). Equation \eqref{alw_latent} comes from the fact that, for all $t$, $P^{ex}_t(\mathbf{A}^-)=P^{ex}_\infty(\mathbf{A}^-)$, because $\mathbf{A}^-$ is the portion of $\mathbf{A}$ that never leaves the latent space. Equation \eqref{by_constr} comes from the updating procedure described in Section \ref{update}, from the countable additivity of $P$, and from $\mathbf{A}^+=(\mathbf{A}^+ \cap \Omega^+_t)\sqcup(\mathbf{A}^+ \cap \Omega^-_t)$. 

Now, notice that the limit as $t \rightarrow \infty$ of $\mathbf{A}^+ \cap \Omega^-_t$ is
$$\mathbf{A}^+ \cap \bigcap_{t \in \mathbb{N}_0}\Omega^-_t=\mathbf{A}^+ \cap \Omega^- =\emptyset,$$
where the last equality is by construction. Then, by the continuity of $P$ we have that 
$$\lim_{t \rightarrow \infty} d_{ETV}(P^{ex}_t,P^{ex}_\infty)=2\lim_{t \rightarrow \infty}  P(\mathbf{A}^+ \cap \Omega^-_t)=2P(\emptyset)=0,$$
which concludes the proof.
\end{proof}

\begin{proof}[Proof of Proposition \ref{prop10}]
Suppose that $P^{ex}_\infty(\Omega)\neq 1$. This of course implies that $P^{ex}_\infty(\Omega)< 1$ because, from the definition of extended probabilities, $P^{ex}_\infty(A)\in [-1,1]$, for all $A \in \mathcal{F}$. Then, this means that there exists a set $A \in \mathcal{F}$ such that $P^{ex}_\infty(A) \leq 0$, which implies $A \subset \Omega^-_\infty$. But we know that, if $\Omega^+_t \uparrow \Omega$, then $\Omega^-_t \downarrow \Omega^-_\infty=\emptyset$, which contradicts $A \subset \Omega^-_\infty$. 
\end{proof}

\begin{proof}[Proof of Theorem \ref{lepc}]
Pick any lower extended probability $\underline{P}^{ex}$ and suppose there exists $\emptyset \neq \mathcal{P}^{ex} \subset \Delta^{ex}(\Omega,\mathcal{F})$ such that $\underline{P}^{ex}(A)=\inf_{P^{ex} \in \mathcal{P}^{ex}}P^{ex}(A)$, for all $A \in \mathcal{F}$. Now, by Theorem \ref{alw_coh}, we know that $P^{ex}$ is coherent, for all $P^{ex} \in \mathcal{P}^{ex}$. This means that for all $\{B_j\}_{j=1}^n \subset \mathscr{B}^\prime$, all $\{s_j\}_{j=1}^n \subset \mathbb{R}$, and all $P^{ex} \in \mathcal{P}^{ex}$,
\begin{align}\label{interm_coh}
\sup_{\omega\in\Omega}\left\lbrace{\sum_{j=1}^n s_j \left[ \mathbbm{1}_{B_j}(\omega) - P^{ex}(B_j) \right]}\right\rbrace \geq 0.
\end{align}
Notice that we use $\mathscr{B}^\prime$ because that is the sigma-algebra generated by the events we can enter a bet about. Then, \eqref{interm_coh} entails that
\begin{align}\label{interm_coh2}
\sup_{\omega\in\Omega} \left[\sum_{j=1}^n s_j \mathbbm{1}_{B_j}(\omega)\right] \geq \sum_{j=1}^n s_j  P^{ex}(B_j).
\end{align}
The inequality in \eqref{interm_coh2} implies that 
\begin{align}\label{interm_coh3}
\begin{split}
\sup_{\omega\in\Omega} \left[\sum_{j=1}^n s_j \mathbbm{1}_{B_j}(\omega)\right] &\geq \inf_{P^{ex} \in \mathcal{P}^{ex}} \sum_{j=1}^n s_j  P^{ex}(B_j)\\
&\geq \sum_{j=1}^n s_j  \inf_{P^{ex} \in \mathcal{P}^{ex}} P^{ex}(B_j)\\
&=\sum_{j=1}^n s_j  \underline{P}^{ex}(B_j).
\end{split}
\end{align}
The results in \eqref{interm_coh3} hold if and only if $\sup_{\omega\in\Omega}\{\sum_{j=1}^n s_j \left[ \mathbbm{1}_{B_j}(\omega) - \underline{P}^{ex}(B_j) \right]\} \geq 0$. The claim follows.
\end{proof}


\begin{proof}[Proof of Corollary \ref{cor_coh}]
Pick a generic lower extended probability $\underline{P}^{ex}$, and suppose $\text{core}(\underline{P}^{ex})$ is nonempty. Then, $\underline{P}^{ex}$ avoids sure loss by Theorem \ref{lepc}. 
\end{proof}


\begin{proof}[Proof of Proposition \ref{core_cpct}]
We first show that core$(\underline{P}^{ex})$ is convex. Pick any $P^{ex}_1,P^{ex}_2$ in the core of $\underline{P}^{ex}$, any $\alpha \in (0,1)$, and any $A \in \mathcal{F}$. We have
$$\alpha P_1^{ex}(A)+(1-\alpha)P_2^{ex}(A) \geq \alpha \underline{P}^{ex}(A) + (1-\alpha) \underline{P}^{ex}(A)=\underline{P}^{ex}(A),$$
so $\alpha P_1^{ex}+(1-\alpha)P_2^{ex} \in \text{core}(\underline{P}^{ex})$.

We then show that core$(\underline{P}^{ex})$ is  weak$^\star$-compact. Recall that, in the weak$^\star$ topology, a net $(P^{ex}_\alpha)_{\alpha \in I}$ converges to $P^{ex}$ if and only if $P^{ex}_\alpha(A) \rightarrow P^{ex}(A)$, for all $A \in \mathcal{F}$. This proof is similar to the proof of \cite[Proposition 3]{marinacci_ambig}, where the authors prove the same claim for the core of a bounded game. Pick any $P^{ex} \in \text{core}(\underline{P}^{ex})$, and let $k:=2\sup_{A \in \mathcal{F}} |\underline{P}^{ex}(A)|$. For any $A \in \mathcal{F}$, it holds that $P^{ex}(A) \geq \underline{P}^{ex}(A) \geq -k$. On the other hand, for all $A \in \mathcal{F}$, we have that
$$P^{ex}(A)=P^{ex}(\Omega)-P^{ex}(A^c) \leq \underline{P}^{ex}(\Omega) - \underline{P}^{ex}(A^c) \leq 2 \sup_{A \in \mathcal{F}} |\underline{P}^{ex}(A)|.$$
This implies that $|P^{ex}(A)|\leq k$, for all $A \in \mathcal{F}$. By \cite[Page 94]{dunford}, we have that
\begin{equation}\label{eq_proof_core}
\|P^{ex}\|:=\sup \sum_{j=1}^n \left| P^{ex}(A_j)-P^{ex}(A_{j-1}) \right| \leq 2k,
\end{equation}
where the supremum is taken over all finite chains $\emptyset=A_0 \subset A_1 \subset \cdots \subset A_n=\Omega$. Then, \eqref{eq_proof_core} implies that 
$$\text{core}(\underline{P}^{ex}) \subset \left\lbrace{P^{ex} \in \Delta^{ex}(\Omega,\mathcal{F}):\|P^{ex}\| \leq 2k}\right\rbrace.$$
By the Alaoglu Theorem \cite[Theorem 2, Page 424]{dunford}, we know that $\{P^{ex} \in \Delta^{ex}(\Omega,\mathcal{F}):\|P^{ex}\| \leq 2k\}$ is weak$^\star$-compact. Hence, to complete the proof, we are left to show that core$(\underline{P}^{ex})$ is  weak$^\star$-closed. Let then $(P^{ex}_\alpha)_{\alpha \in I}$ be a net in core$(\underline{P}^{ex})$ that weak$^\star$-converges to $P^{ex} \in \Delta^{ex}(\Omega,\mathcal{F})$. Using the properties of the weak$^\star$ topology, it is easy to see that $P^{ex} \in \text{core}(\underline{P}^{ex})$. Hence, $\text{core}(\underline{P}^{ex})$ is weak$^\star$-closed. The claim follows.
\end{proof}

\begin{proof}[Proof of Proposition \ref{prop12}]
Equation \eqref{eq8} comes from the fact that, for all $t$, $${P}_{t+1}^{ex}(E^+_{t+1,\omega})=|{P}_t^{ex}(E^-_{t,\omega})|,$$
that is, ${P}_{t+1}^{ex}(E^+_{t+1,\omega})=-{P}_t^{ex}(E^-_{t,\omega})$, and that
$$-\overline{P}_t^{ex}(E^-_{t,\omega}) \leq -{P}_t^{ex}(E^-_{t,\omega}) \leq -\underline{P}_t^{ex}(E^-_{t,\omega}). $$
\end{proof}

\section{Application to opinion dynamics}\label{opdin}
This example comes from the model in \cite{allah}. There, opinion dynamics between a persuaded and a persuading agents is studied, in particular when the persuaded agent evaluates new information in a way that is consistent with her own preexisting belief. In this example we are going to adoperate extended probabilities to model the boomerang effect. This phenomenon corresponds to the empirical observation that sometimes persuasion yields the opposite effect: the persuaded agents moves her opinion away from the opinion of the persuading agent. That is, she enforces her old opinion. 

In \cite{allah}, the authors assume that the state of the world does not change, that the agents are aware of this fact, and that the persuaded agent changes her opinion only under the influence of the opinion of the persuading agent. They model this dynamic in a linear fashion. The first iteration is the following
\begin{align}\label{eq30}
\hat{P}_1(\{\omega_k\})=\epsilon P_0(\{\omega_k\})+(1-\epsilon)Q(\{\omega_k\}), \quad \epsilon \in [0,1], 
\end{align}
where $\omega_k \in \Omega$, a finite sample space, and $P_0$ and $Q$ are regular probability measures that represent the initial opinions of the persuaded and the persuading agents, respectively. $Q$ is not indexed to time because they make the simplifying assumption that the persuading agent does not change his mind: he tries to persuade the other agent of the same thing at every iteration. $\epsilon$ is a weight, and several qualitative factors contribute to its subjective assessment: egocentric attitude of the persuaded agent, the fact that the persuaded agent has access to internal reasons for choosing her opinion, while she is not aware of the internal reasons of the persuading agent, and many more. The authors relate $\epsilon$ to the credibility of the persuading agent: the higher $\epsilon$, the less credible he is. The successive iterations are modeled as follows
\begin{align}\label{eq38}
\hat{P}_{t+1}(\{\omega_k\})=\epsilon \hat{P}_t(\{\omega_k\})+(1-\epsilon)Q(\{\omega_k\}), 
\end{align}
for $t \geq 1$. This means that at every iteration, the persuading agent tries to shift the persuaded agent's opinion closer to his own. This continues until either the persuaded agent is content with her opinion (and hence does not further change her beliefs), or the persuading agent completely convinces the other agent.

Now, one way to model the boomerang effect is to consider $\epsilon>1$. This conveys the idea that the persuading agent has an extremely low credibility. This results in the updated opinion of the persuaded agent to be an extended probability measure. Indeed, for $\epsilon>1$, $\hat{P}_{t+1}(\{\omega_k\})$ is negative whenever $\epsilon \hat{P}_t(\{\omega_k\}) < |1-\epsilon| Q(\{\omega_k\})$. In \cite{allah} the authors consider the induced regular probability measure to avoid working with extended probabilities.

We modify slightly the linear model in \cite{allah}. The main differences are  three: we allow the  use of extended probabilities, we describe a state space that is divided in latent and known, and whose composition is gradually discovered as the time passes by, and we do not let  $\epsilon$ be a free parameter. It depends on both the element $\omega_k$ of the state space  we examine and on the iteration $t$ we are considering.

Mathematically, we can describe the low credibility of the persuading agent through hidden states:
the persuaded agent may think that she does not know the state space well enough, that is, that there are some hidden portions of $\Omega$ she is not (yet) aware of.


We assume that the state space $\Omega=\{\omega_1,\ldots,\omega_N\}$ is a finite set (a common simplifying assumption in opinion dynamics). Notice that the finest possible partition of $\Omega$ is $\mathcal{E}=\{\{\omega_j\}\}_{j=1}^N$. At the beginning of the interaction between persuading and persuaded agents, the former is fully aware of the composition of the state space, while the latter is only aware of the composition of $\Omega^+_0 \subsetneq \Omega$, but suspects that the actual state space is larger. This corresponds to having suspects on the persuading agent hiding some pieces of information. In particular, she correctly guesses that the true state space is $\Omega$. This correct guess is without loss of generality for our analysis: we are in the $\Omega_t^+ \uparrow \Omega$ case; we also assume that the discovering procedure is equivalent to an urn without replacement. She then defines an extended probability on $\Omega$ the way we explained in section \ref{main}. That is, she specifies the oracle probability distribution $P$ on the whole $\Omega$ and then flips the sign to the probabilities of the latent events: $P^{ex}_0(\{\omega_k\})=P(\{\omega_k\})$ if $\omega_k \in \Omega^+_0$, and $P^{ex}_0(\{\omega_k\})=-P(\{\omega_k\})$ if $\omega_k \not\in \Omega^+_0$.  

To obtain the influenced extended probability at any time $t \geq 0$, we modify slightly equation \eqref{eq38} to get
\begin{align}\label{eq33}
\hat{P}^{ex}_{t}(\{\omega_k\})=\epsilon_{k,t} P^{ex}_t(\{\omega_k\})+(1-\epsilon_{k,t})Q(\{\omega_k\}).
\end{align} 
The nonnegative $\epsilon_{k,t}$'s have to be chosen such that $\hat{P}^{ex}_{t}$ is an extended probability measure, that is, $\hat{P}^{ex}_{t}(\{\omega_k\}) \in [-1,1]$ for all $k$, for all $t$, and 
$$\hat{P}^{ex}_{t}(\Omega)=\sum_{k=1}^N \hat{P}^{ex}_{t}(\{\omega_k\}) \leq 1.$$
Here, $\hat{P}^{ex}_{t}$ denotes the influenced extended probability at time $t$. Notice that $\hat{P}^{ex}_{t}$ is analytically similar to an $\epsilon$-contaminated probability measure. There are of course two major differences: for some $\omega_k$, $\epsilon_{k,t}$ is greater than $1$, and also one of the elements of the mixture is an extended probability measure (rather than a regular one). 
Notice also that $Q$ is not indicized to time because we too make the simplifying assumption that the persuading agent does not change his mind. Another characteristic worth noting is that in our model, at every iteration $t$, the persuaded agent combines her {updated} belief (expressed via the extended probability $P^{ex}_t$) with the other agent's belief to obtain $\hat{P}^{ex}_t$. This is different from the model in \cite{allah} where at every iteration $t$ the persuaded agent combines her {influenced} belief at iteration $t-1$ (expressed through $\hat{P}_{t-1}$) with the other agent's belief to obtain $\hat{P}_t$. This because in their model the world does not change, so the only way of describing an opinion dynamics is the one the authors illustrate.

The persuaded agent updates $P^{ex}_t$ as specified in section \ref{main}. That is, when at time $t$ she observes $\omega_k$ that used to belong to the latent space, $P^{ex}_t(\{\omega_k\})=|P^{ex}_{t-1}(\{\omega_k\})|$, while the extended probabilities for $\omega_s \neq \omega_k$ are kept constant. Let us be more precise about the differences with equation \eqref{eq30}; $\epsilon_{k,t}$ depends on both $k$ and $t$. The persuaded agent has a different perception of the opinion of the persuading agent depending on whether she is not sure the topic they are debating about belongs to the state space $\Omega$, so for $\omega_k$ in the latent space, $\epsilon_{k,t}$ is greater than $1$. In addition, as time passes by, the hidden elements of the state space become known, so (part of) the credibility of the persuading agent is restored. The $\epsilon_{k,t}$ associated with $\omega_k$ observed at time $t$ becomes smaller than $1$, for all $\omega_k$. 

Notice that, for $t \geq N$, $P^{ex}_t$ is a regular probability measure, because the persuaded agent discovers the composition of the whole state space, so the latent space shrinks to the empty set.

Throughout this section we made the tacit assumption that $\mathcal{P}$, the set of probability measures on $\Omega$ that induces the set of extended probabilities $\mathcal{P}^{ex}_0$ at time $t=0$, is the singleton $\{P\}$, so that $\mathcal{P}^{ex}_0=\{P^{ex}_0\}$. This simplifying assumption can be dropped, and the analysis stays the same. The only difference is that we have to repeat it for all the elements of $\mathcal{P}^{ex}_t$, for all $t \in \mathbb{N}_0$. Every set $\mathcal{P}^{ex}_t$ of extended probabilities induces a set $\hat{\mathcal{P}}^{ex}_t$ of influenced extended probabilities. For all $t \geq N$,  $\mathcal{P}^{ex}_t$ is a set of regular probability measures, and $\hat{\mathcal{P}}^{ex}_t$ is a set of influenced regular probability measures.

\section{DeFinettian interpretation of subjective probability}\label{definetti}

De Finetti admits that it might have been better to adopt the seemingly more general approach of Ramsey and Savage of defining bets whose payoffs are in utils \cite[Page 79]{definetti1}: 

\textit{
The formulation [...] could be made watertight [...] by working in terms of the utility instead of with monetary value. This would undoubtedly be the best course from
the theoretical point of view, because one could construct, in an integrated fashion,
a theory of decision-making [...] whose meaning would be unexceptionable from
an economic viewpoint, and which would establish simultaneously and in parallel
the properties of probability and utility on which it depends. 
}

Nevertheless, he found ``other reasons for preferring'' the money
bet approach \cite[Page 81]{definetti1}:

\textit{
The main motivation lies in being able to refer, in a natural way to combinations of
bets, or any other economic transactions, understood in terms of monetary value
(which is invariant). If we referred ourselves to the scale of utility, a transaction
leading to a gain of amount $S$ if the event $E$ occurs would instead appear as
a variety of different transactions, depending on the outcome of other random
transactions. These, in fact, cause variations in one’s fortune, and therefore in
the increment of utility resulting from the possible additional gain $S$: conversely,
suppose that in order to avoid this one tried to consider bets, or economic transactions, expressed, let us say, in ``utiles'' (units of utility, definable as the increment
between two fixed situations). In this case, it would be practically impossible to
proceed with the transactions, because the real magnitudes in which they have to
be expressed (monetary sums or quantities of goods, etc.) would have to be adjusted to the continuous and complex variations in a unit of measure that nobody
would be able to observe. 
}

\bibliographystyle{plain}
\bibliography{References}
\end{document}